\documentclass[a4paper, 11pt]{amsart} 
\usepackage[
  margin=1.9cm,
  includefoot,
  footskip=30pt,
]{geometry}
\usepackage{amsmath,amssymb,amscd}
\usepackage{amsfonts}
\usepackage[english]{babel}
\usepackage[leqno]{amsmath}
\usepackage{amssymb,amsthm}
\usepackage{mathrsfs} 
\usepackage{amscd}
\usepackage{enumerate}
\usepackage{epsfig}
\usepackage{relsize}
\usepackage{layout}
\usepackage[usenames,dvipsnames]{xcolor}
\usepackage[backref=page]{hyperref}
\usepackage{tikz}
\hypersetup{
 colorlinks,
 citecolor=Green,
 linkcolor=Red,
 urlcolor=Blue}
\usepackage[matrix,arrow,tips,curve]{xy}
\input{xy}
\xyoption{all}
\definecolor{darkgreen}{rgb}{0.0, 0.7, 0.0}
\definecolor{purple}{rgb}{0.5, 0.0, 0.5}
\definecolor{red}{rgb}{0.8, 0.2, 0.0}

\newtheorem{thm}{Theorem}[section]
\newtheorem{bthm}{Theorem}

\newtheorem{lemma}[thm]{Lemma}

\newtheorem{prop}[thm]{Proposition}
\newtheorem{cor}[thm]{Corollary}

\numberwithin{equation}{section}
\theoremstyle{definition}

\newtheorem{setup}[thm]{Setup}

\theoremstyle{remark}
\newtheorem{remark}[thm]{Remark}

\newcommand{\Z}{\mathbb{Z}}

\newcommand{\Q}{\mathbb{Q}}

\def \P{\mathbb{P}}

\def \F{\mathcal F}

\def\I{{\mathcal J}}
\def \L{\mathcal L}
\def \E{\mathcal E}
\def \G{\mathcal G}

\def\O{\mathcal O}

\def\M0{\mathcal M^0}

\def\mapright#1{\smash{\mathop{\longrightarrow}\limits^{#1}}}

\DeclareMathOperator{\Coker}{{Coker}}

\DeclareMathOperator{\Ker}{{Ker}}

\newcommand{\cC}{{\mathcal C}}

\newcommand{\cK}{{\mathcal K}}

\newcommand{\cQ}{{\mathcal Q}}

\newcommand{\rk}{\operatorname{rank}}

\newcommand{\uc}{\operatorname{Uc}}

\begin{document}

\title[A lower bound on the Ulrich complexity of hypersurfaces]{A lower bound on the Ulrich complexity of hypersurfaces}

\author[A.F. Lopez, D. Raychaudhury]{Angelo Felice Lopez and Debaditya Raychaudhury}

\address{\hskip -.43cm Angelo Felice Lopez, Dipartimento di Matematica e Fisica, Universit\`a di Roma
Tre, Largo San Leonardo Murialdo 1, 00146, Roma, Italy. e-mail {\tt angelo.lopez@uniroma3.it}}

\address{\hskip -.43cm Debaditya Raychaudhury, Department of Mathematics, University of Arizona, 617 N Santa Rita Ave., Tucson, AZ 85721, USA. email: {\tt draychaudhury@arizona.edu}}

\thanks{The first author was partially supported by the GNSAGA group of INdAM and by the PRIN ``Advances in Moduli Theory and Birational Classification''. The second author was partially supported by an AMS-Simons Travel Grant.}


\begin{abstract} 
We give a lower bound on the Ulrich complexity of hypersurfaces of dimension $n \ge 6$. The bound improves the result in 
\cite{bes} for $n \le 78$.
\end{abstract}

\maketitle

\section{Introduction}

Let $X \subset \P^{n+1}$ be a smooth hypersurface of degree $d$. While the study of arithmetically Cohen-Macaulay bundles on $X$ is a classical one, in recent years the special class of Ulrich bundles has attracted attention. A bundle $\E$ on $X$ is called Ulrich if $H^i(\E(-p))=0$ for $i \ge 0$ and $1 \le p \le n$ (for general facts on them see for example \cite{cmp, es, b2}). Since a hypersurface $X$ always carries an Ulrich bundle (of high rank) by \cite{hub}, a more refined invariant is the {\it Ulrich complexity}, namely 
$$\uc(X) = \min\{\rk \E, \E \ \hbox{Ulrich bundle on} \ X\}.$$
In low degree, we know (see for example \cite{b2}) that $\uc(X)=1$ if $d=1$ and $\uc(X)=2^{\lfloor \frac{n-1}{2} \rfloor}$ if $d=2$. When $d \ge 3$, the Buchweitz, Greuel and Schreyer's conjecture \cite{bgs}, would imply that $\uc(X) \ge 2^{\lfloor \frac{n-1}{2} \rfloor}$ and also that $\uc(X) \ge 2^{\lfloor \frac{n+1}{2} \rfloor}$ when $X$ is general \cite{rt1} (see also \cite{e}). Aside from several special cases \cite{b1, b2, ch, cfk, fk}, the best known lower bound was recently shown in \cite[Thm.~3.1]{bes}: $\uc(X) \ge \sqrt{n+2}-1$. Also, it was proved in \cite{lr3, rt1, rt2} (the second one is for aCM bundles) that $\uc(X) \ge 4$ if $n \ge 5$ or if $n=3, 4$ and $X$ is general.  

Along these lines, our main result is as follows. 

\begin{bthm}
\label{main}

\hskip 3cm

Let $X \subset \P^{n+1}$ be a smooth hypersurface of degree $d \ge 3$. Then the following lower bounds hold:
\begin{itemize}
\item[(i)] $\uc(X) \ge 6$, if either $n \in \{7, 8\}$, or $n=6$ and $X$ is very general.
\item[(ii)] $\uc(X) \ge 8$, if either $n \ge 9$, or $n=8$ and $X$ is very general. 
\end{itemize} 
\end{bthm}
We point out that the above theorem improves the bound in \cite[Thm.~3.1]{bes} for $n \le 78$.

Here is a brief summary of the paper. In Section \ref{dege} we study the invariants and the geometry of the degeneracy locus of two sections of a globally generated bundle. Section \ref{gen} is dedicated to proving some useful general facts about Ulrich bundles, Section \ref{inv} is about Chern classes of Ulrich bundles on hypersurfaces. Finally, in Section \ref{pf}, we prove the above theorem. In the appendix we perform the necessary computations needed.

Throughout the paper we work over the complex numbers. Moreover, we establish the following notation and conventions.

Given $X \subset \P^N$ and $i \in \{1,\ldots, n-1\}$, we denote by $X_i$ the intersection of $X$ with $n-i$ general hyperplanes. We say that $X$ is {\it subcanonical} if $-K_X=i_XH$ for some $i_X \in \Z$.

We use the convention $\binom{\ell}{m}=\frac{\ell (\ell-1)\ldots (\ell-m+1)}{m!}$ for $\ell, m \in \Z, m \ge 1$.

\section{A useful degeneracy locus}
\label{dege}

In the proof of the main theorem, a crucial role will be played by a suitable degeneracy locus. In this section we will introduce it and study its properties.

The following will henceforth be fixed in this section.

\begin{setup}
\label{set}

\hskip 3cm

\begin{itemize} 
\item $X \subset \P^N$ is a smooth irreducible variety of dimension $n \ge 3$.
\item $r$ is an integer such that $\frac{n+1}{2} \le r \le n+1$. 
\item $\E$ is a rank $r$ globally generated bundle on $X$ with $\det \E = \O_X(D)$.
\item $V \subset H^0(\E)$ is a general subspace of dimension $2$, giving rise to
$\varphi : V \otimes \O_X \to \E$.
\item $Z =D_1(\varphi)=\{x \in X : \rk \varphi(x) \le 1\}$ is the corresponding degeneracy locus. 
\end{itemize} 
\end{setup}

\begin{lemma} 
\label{nldg}

Notation as in Setup \ref{set}. If $Z \ne \emptyset$, then $Z$ is a smooth subvariety of $X$ of pure dimension $n+1-r$ and $[Z]=c_{r-1}(\E) \in H^{2r-2}(X,\Z)$. We have
\begin{equation}
\label{k2}
(K_Z-(K_X+D)_{|Z})^2=0
\end{equation}
\begin{equation}
\label{ch2}
(r-2)c_2(Z)=(r-2)c_2(X)_{| Z}-(r-2)c_2(\E)_{| Z}+(K_Z-{K_X}_{|Z})[(r-2){K_X}_{|Z}+(r-1)D_{|Z}]-D_{|Z}^2
\end{equation}
and a resolution
\begin{equation}
\label{ea1}
0 \to F_{r-1} \to \ldots \to F_1 \to \I_{Z/X} \to 0
\end{equation}
where $F_i=(\Lambda^{r-1-i} \E \otimes \O_X(-D))^{\oplus i}, 1 \le i \le r-1$.
\end{lemma}
\begin{proof}
We first prove that
\begin{equation}
\label{vuo}
D_0(\varphi)=\emptyset.
\end{equation}
In fact, we get by \cite[Statement (folklore)(i), \S 4.1]{ba} that if $D_0(\varphi) \ne \emptyset$, then it has pure codimension $2r$, a contradiction. This proves \eqref{vuo}. Since we are assuming that $Z \ne \emptyset$, it follows by \cite[Statement (folklore)(i), \S 4.1]{ba} and \eqref{vuo}, that $Z$ is smooth of pure codimension $r-1$ and then $[Z]=c_{r-1}(\E)$. Set 
$$\cK = \Ker(\varphi_{|Z}), \ \ \cC = \Coker(\varphi_{|Z}) \ \hbox{and} \ \cQ = \Ker(\E_{|Z} \to \cC)$$
so that we have two exact sequences of vector bundles on $Z$,
\begin{equation}
\label{pr}
0 \to \cQ \to \E_{| Z} \to \cC \to 0
\end{equation}
\begin{equation}
\label{se}
0 \to \cK \to \O_Z^{\oplus 2} \to \cQ \to 0.
\end{equation}
It follows, by \eqref{vuo} and \cite[(5.1)]{fp}, that  $\cC$ (respectively $\cK$) is a vector bundle on $Z$ of rank $r-1$ (respectively $1$) and that $N_{Z/X} \cong \cK^* \otimes \cC$. Therefore also $\cQ$ is a line bundle on $Z$ and we get by \eqref{pr} and \eqref{se} that
$$c_1(\cQ)+c_1(\cC)=D_{|Z}, c_1(\cK)+c_1(\cQ)=0$$
and therefore
$$c_1(\cC)=c_1(N_{Z/X} \otimes \cK)=K_Z-{K_X}_{|Z}+(r-1)c_1(\cK).$$
Using these we find 
\begin{equation}
\label{c_1k}
(r-2)c_1(\cK)=[(K_X+D)_{|Z}-K_Z]
\end{equation}
and then
$$(r-2)c_1(\cC)=[(K_X+(r-1)D)_{|Z}-K_Z].$$
Now, to prove \eqref{k2} and \eqref{ch2}, we will use the exact sequence
\begin{equation}
\label{tg}
0 \to T_Z \to {T_X}_{|Z} \to N_{Z/X} \to 0.
\end{equation}
Since $\cQ$ is a line bundle, \eqref{se} shows that 
$$0 = c_2(\O_Z^{\oplus 2})=c_1(\cK)^2$$
so that \eqref{k2} holds by \eqref{c_1k} and then \eqref{ch2} follows by computing Chern classes in \eqref{pr}, \eqref{tg} and using \eqref{k2}. Finally, to see \eqref{ea1}, observe that $Z= D_1(\varphi)=D_1(\varphi^*)$ where $\varphi^* : \E^* \to V^*\otimes \O_X$. Since $Z$ has the expected codimension, the Eagon-Northcott complex gives a resolution \cite[Thm.~B.2.2(iii)]{laz1}
$$0 \to F_{r-1} \to \ldots \to F_1 \to \I_{Z/X} \to 0$$
where $F_i=S^{i-1} V^* \otimes \Lambda^{i+1} \E^* \cong (\Lambda^{r-1-i} \E \otimes \O_X(-D))^{\oplus i}, 1 \le i \le r-1$. 
\end{proof}

Next, we check non-emptiness and irreducibility of $Z$ as above.

\begin{lemma}
\label{nonvuoto}

Notation as in Setup \ref{set}. We have:
\begin{itemize}
\item[(i)] $Z \ne \emptyset$ if $c_{r-1}(\E) \ne 0$.
\end{itemize}
Moreover, if, in addition, $r \le n$ and $Z \ne \emptyset$, then $Z$ is smooth and irreducible if one of the following holds:
\begin{itemize}
\item[(ii)] $\E$ is $(n-r)$-ample, or
\item[(iii)] $H^i(\Lambda^{i+1} \E^*)=0$ for $1 \le i \le r-1$.
\end{itemize}
\end{lemma}
\begin{proof}
If $Z = \emptyset$, then the morphism $\varphi$ has constant rank $2$ and therefore we get an exact sequence
$$0 \to V \otimes \O_X \to \E \to \F \to 0$$
where also $\F$ is a vector bundle of rank $r-2$. But then $c_{r-1}(\E)=c_{r-1}(\F)=0$, a contradiction. Therefore $Z \ne \emptyset$ and (i) is proved. To see that $Z$ is smooth and irreducible it is enough, by Lemma \ref{nldg}, to prove that $Z$ is connected. Now, under the hypothesis in (ii), the connectedness follows by \cite[Thm.~6.4(a)]{t}. Under the hypothesis in (iii), since $\Lambda^{i+1} \E^* \cong \Lambda^{r-1-i} \E \otimes \O_X(-D)$, we deduce by \eqref{ea1} and 
\cite[Prop.~B.1.2(i)]{laz1} that $H^1(\I_{Z/X})=0$, hence again $Z$ is connected. 
\end{proof}

\section{Generalities on Ulrich vector bundles}
\label{gen}

We will often use the following, mostly well-known, properties of Ulrich bundles.

\begin{lemma} 
\label{ulr}
Let $X \subseteq \P^N$ be a smooth irreducible variety of dimension $n$, degree $d$ and let $\E$ be a rank $r$ Ulrich bundle. We have:
\begin{itemize}
\item[(i)] $\E$ is globally generated. 
\item[(ii)] $\E^*(K_X+(n+1)H)$ is Ulrich.
\item[(iii)] $\E$ is aCM.
\item[(iv)] $c_1(\E) H^{n-1}=\frac{r}{2}[K_X+(n+1)H] H^{n-1}$.
\item [(v)] $\E_{|Y}$ is Ulrich on a smooth hyperplane section $Y$ of $X$.
\item[(vi)] $\det \E$ is globally generated and it is not trivial, unless $(X, H, \E) = (\P^n, \O_{\P^n}(1), \O_{\P^n}^{\oplus r})$.
\item[(vii)] If $n \ge 2$, then $c_2(\E) H^{n-2}=\frac{1}{2}[c_1(\E)^2-c_1(\E) K_X] H^{n-2}+\frac{r}{12}[K_X^2+c_2(X)-\frac{3n^2+5n+2}{2}H^2] H^{n-2}$.
\item[(viii)] $\chi(\E(m))= rd\binom{m+n}{n}$.
\item[(ix)] If $n=3$, then
$$c_3(\E)=2r(d-\chi(\O_X))+c_1(\E)c_2(\E)-\frac{1}{3}c_1(\E)^3+\frac{1}{2}K_X(c_1(\E)^2-2c_2(\E))-\frac{1}{6}(K_X^2+c_2(X))c_1(\E).$$
\item[(x)] If $n=4$, then
$$\begin{aligned}[t] 
c_4(\E) = & -6r(d-\chi(\O_X))-\frac{1}{4}K_Xc_2(X)c_1(\E)+\frac{1}{4}[K_X^2+c_2(X)][c_1(\E)^2-2c_2(\E)]\\
& -\frac{1}{2}K_X[c_1(\E)^3-3c_1(\E)c_2(\E)+3c_3(\E)]+\frac{1}{4}[c_1(\E)^4-4c_1(\E)^2c_2(\E)+4c_1(\E)c_3(\E)+2c_2(\E)^2]
\end{aligned}$$
\item[(xi)] If $n=5$, then
$$\begin{aligned}[t] 
c_5(\E) = & 24r(d-\chi(\O_X))-\frac{1}{5}c_1(\E)^5+c_1(\E)^3c_2(\E)-c_1(\E)^2c_3(\E)-c_1(\E)c_2(\E)^2
+c_1(\E)c_4(\E)\\
& +c_2(\E)c_3(\E)+\frac{1}{2}(c_1(\E)^2-2c_2(\E))c_2(X)K_X\\
& +\frac{1}{30}c_1(\E)(K_X^4-4K_X^2c_2(X)+K_Xc_3(X)-3c_2(X)^2+c_4(X))\\
& +\frac{1}{2}(c_1(\E)^4-4c_1(\E)^2c_2(\E)+4c_1(\E)c_3(\E)+2c_2(\E)^2-4c_4(\E))K_X\\
& -\frac{1}{3}(K_X^2+c_2(X))(c_1(\E)^3-3c_1(\E)c_2(\E)+3c_3(\E)).
\end{aligned}$$
\item[(xii)] If $n=6$, then
$$\begin{aligned}[t] 
& c_6(\E) = -120r(d-\chi(\O_X))-\frac{1}{12}c_1(\E)(-K_X^3c_2(X)+3K_Xc_2(X)^2-K_X^2c_3(X)-K_Xc_4(X))\\
& -\frac{1}{12}(K_X^4c_1(\E)^2-4K_X^2c_2(X)c_1(\E)^2-3c_2(X)^2c_1(\E)^2+K_Xc_3(X)c_1(\E)^2+c_4(X)c_1(\E)^2\\
& -2K_X^4c_2(\E)+8K_X^2c_2(X)c_2(\E)+6c_2(X)^2c_2(\E)-2K_Xc_3(X)c_2(\E)-2c_4(X)c_2(\E))\\
& -\frac{5}{6}K_Xc_2(X)(c_1(\E)^3-3c_1(\E)c_2(\E)+3c_3(\E))\\
& +\frac{5}{12}(K_X^2+c_2(X))(c_1(\E)^4-4c_1(\E)^2c_2(\E)+2c_2(\E)^2+4c_1(\E)c_3(\E)-4c_4(\E))\\
& -\frac{1}{2}K_X(c_1(\E)^5-5c_1(\E)^3c_2(\E)+5c_1(\E)c_2(\E)^2+5c_1(\E)^2c_3(\E)-5c_2(\E)c_3(\E)-5c_1(\E)c_4(\E)+5c_5(\E))\\
& +\frac{1}{6}c_1(\E)^6-c_1(\E)^4c_2(\E)+\frac{3}{2}c_1(\E)^2c_2(\E)^2-\frac{1}{3}c_2(\E)^3+c_1(\E)^3c_3(\E)-2c_1(\E)c_2(\E)c_3(\E)+\frac{1}{2}c_3(\E)^2\\
& -c_1(\E)^2c_4(\E)+c_2(\E)c_4(\E)+c_1(\E)c_5(\E).
\end{aligned}$$
\item[(xiii)] If $n=7$, then
$$\begin{aligned}[t] 
& c_7(\E)=720r(d-\chi(\O_X))\\
& +\frac{1}{2}K_X(c_1(\E)^6-6c_1(\E)^4 c_2(\E)+9c_1(\E)^2 c_2(\E)^2-2c_2(\E)^3+6c_1(\E)^3 c_3(\E)-12c_1(\E)c_2(\E)c_3(\E)\\
& \hskip 1.4cm +3c_3(\E)^2-6c_1(\E)^2 c_4(\E)+6c_2(\E)c_4(\E)+6c_1(\E)c_5(\E)-6c_6(\E))\\
& -\frac{1}{2}(K_X^2+c_2(X))(c_1(\E)^5-5c_1(\E)^3 c_2(\E)+5c_1(\E)c_2(\E)^2+5c_1(\E)^2 c_3(\E)-5c_2(\E)c_3(\E)\\
& \hskip 3.1cm -5c_1(\E)c_4(\E)+5c_5(\E))\\
& +\frac{5}{4}K_Xc_2(X)(c_1(\E)^4-4c_1(\E)^2 c_2(\E)+2c_2(\E)^2+4c_1(\E)c_3(\E)-4c_4(\E))\\
& +\frac{1}{6}(K_X^4 c_1(\E)^3-4K_X^2 c_2(X)c_1(\E)^3-3c_2(X)^2 c_1(\E)^3+K_Xc_3(X)c_1(\E)^3+c_4(X)c_1(\E)^3\\
& \hskip .8cm -3K_X^4c_1(\E)c_2(\E)+12K_X^2 c_2(X)c_1(\E)c_2(\E)+9c_2(X)^2 c_1(\E)c_2(\E)-3K_Xc_3(X)c_1(\E)c_2(\E)\\
& \hskip .8cm -3c_4(X)c_1(\E)c_2(\E)+3K_X^4 c_3(\E)-12K_X^2 c_2(X)c_3(\E)-9c_2(X)^2 c_3(\E)+3K_Xc_3(X)c_3(\E)\\
& \hskip .8cm +3c_4(X)c_3(\E))\\
& -\frac{1}{4}K_X(K_X^2 c_2(X)c_1(\E)^2-3c_2(X)^2 c_1(\E)^2+K_Xc_3(X)c_1(\E)^2+c_4(X)c_1(\E)^2-2K_X^2 c_2(X)c_2(\E)\\
& \hskip 1.4cm +6c_2(X)^2 c_2(\E)-2K_Xc_3(X)c_2(\E)-2c_4(X)c_2(\E))\\
& -\frac{1}{84}c_1(\E)(2K_X^6-12K_X^4 c_2(X)+11K_X^2 c_2(X)^2+10c_2(X)^3-5K_X^3 c_3(X)-11K_Xc_2(X)c_3(X)\\
& \hskip 1.8cm -c_3(X)^2-5K_X^2 c_4(X)-9c_2(X)c_4(X)+2K_Xc_5(X)+2c_6(X))\\
& -\frac{1}{7}c_1(\E)^7+c_1(\E)^5 c_2(\E)-2c_1(\E)^3 c_2(\E)^2+c_1(\E)c_2(\E)^3-c_1(\E)^4 c_3(\E)+3c_1(\E)^2 c_2(\E)c_3(\E)\\
& -c_2(\E)^2 c_3(\E)-c_1(\E)c_3(\E)^2+c_1(\E)^3 c_4(\E)-2c_1(\E)c_2(\E)c_4(\E)+c_3(\E)c_4(\E)-c_1(\E)^2 c_5(\E)+c_2(\E)c_5(\E)\\
& +c_1(\E)c_6(\E).
\end{aligned}$$
\end{itemize} 
\end{lemma}
\begin{proof}
See for example \cite[Lemma 3.2]{lr1} for (i)-(v) and (vii)-(viii), \cite[Lemma 2.1]{lo} for (vi) and \cite[Prop.~3.7(b)]{bmpt} for (ix). The formulas (x)-(xiii) follow by using Riemann-Roch on $X$ and $\chi(\E)=rd$ by (vii) (see \cite[Out(17), Out(19), Out(21), Out(23)]{mcc}).
\end{proof}

We have the following easy vanishings for powers of Ulrich bundles.

\begin{lemma}
\label{nuv}

Let $X \subset \P^N$ be a smooth irreducible variety of dimension $n$ and let $\E$ be a rank $r$ Ulrich bundle. Then:
\begin{itemize}
\item[(i)] $H^n(\E^{\otimes j}(l))=0$ for $j \ge 1$ and $l \ge -n$.
\item[(ii)] $H^{n-1}(\E^{\otimes j}(l))=0$ for $j \ge 1$ and $l \ge 1-n$.
\item[(iii)] $H^n((\Lambda^j \E)(l))=0$ for $j \ge 1, l \ge -n$ and $H^{n-1}((\Lambda^j \E)(l))=0$ for $j \ge 1, l \ge 1-n$.
\end{itemize} 
\end{lemma}
\begin{proof}
It is well-known that (iii) follows by (i) and (ii), which we now prove.
If $d=1$ we have that $(X, H, \E) = (\P^n, \O_{\P^n}(1), \O_{\P^n}^{\oplus r})$ by \cite[Prop.~2.1]{es}, hence (i) and (ii) follow. If $d \ge 2$ we prove (i) and (ii) by induction on $j$. The case $j=1$ follows by Castelnuovo-Mumford since $\E$ is $0$-regular. Suppose $j \ge 2$. We have, by \cite[Prop.~2.1]{es}, an exact sequence
$$\O_{\P^N}(-1)^{\oplus \beta_1} \to \O_{\P^N}^{\oplus \beta_0} \to \E \to 0$$
hence, tensoring by $\E^{\otimes (j-1)}(l)$ we get an exact sequence on $X$,
\begin{equation}
\label{su}
0 \to \G \to (\E^{\otimes (j-1)}(l-1))^{\oplus \beta_1} \mapright{\psi} (\E^{\otimes (j-1)}(l))^{\oplus \beta_0} \to \E^{\otimes j}(l) \to 0
\end{equation}
where $\G = \Ker \psi$. If $l \ge -n$, we have that $H^n((\E^{\otimes (j-1)}(l)))=0$ by the inductive hypothesis, hence $H^n(\E^{\otimes j}(l))=0$ by \eqref{su}. This proves (i). If $l \ge 1-n$, we have that $H^{n-1}((\E^{\otimes (j-1)}(l)))=0$ by the inductive hypothesis and $H^n((\E^{\otimes (j-1)}(l-1)))=0$ by (i). Then (ii) follows by \eqref{su} and \cite[Prop.~B.1.2(i)]{laz1}. 
\end{proof}

We will now give some conditions under which $Z$ is connected.

\begin{lemma}
\label{duecasi}

Let $X \subset \P^N$ be a smooth irreducible variety of dimension $n$, degree $d \ge 3$ and let $\E$ be a rank $r$ Ulrich bundle. Suppose that $X$ is subcanonical and $\det \E = \O_X(u)$, for some $u \in \Z$. Assume that one of the following conditions is satisfied:
\begin{itemize}
\item[(a)] $4 \le n \le 7$ and $r=4$, or
\item[(b)] $6 \le n \le 9$ and $r=5$.
\end{itemize}
Let $Z$ be as in Setup \ref{set} and assume that $Z \ne \emptyset$. Then $Z$ is smooth and irreducible.
\end{lemma}
\begin{proof}
Set $K_X=-i_XH$, so that we know, by Lemma \ref{ulr}(iv) and (vi), that $u=\frac{r(n+1-i_X)}{2}>0$. Note that $-i_X \ge 1-n$, for otherwise $i_X \ge n$, hence $X$ is Fano and, as is well-known, this gives $d \le 2$, a contradiction. Also, note that $r$ and $n$ satisfy the conditions in Setup \ref{set}, hence $Z$ is smooth by Lemma \ref{nldg}. The plan is now to apply Lemma \ref{nonvuoto}(iii), hence to show that
\begin{equation}
\label{vala}
H^i(\Lambda^{i+1} \E^*)=0, \ \hbox{for} \ 1 \le i \le r-1.
\end{equation}
For $i=r-1$, we have that $H^{r-1}(\Lambda^r \E^*)=H^{r-1}(-uH)=0$ by Kodaira vanishing. For $i=r-2$, since $\Lambda^{r-1} \E^* \cong \E(-u)$, we have that $H^{r-2}(\Lambda^{r-1} \E^*)=H^{r-2}(\E(-u))=0$ by Lemma \ref{ulr}(iii). For $i=1$, by Serre's duality we have that $H^1(\Lambda^2 \E^*) \cong H^{n-1}((\Lambda^2 \E)(-i_X))=0$ by Lemma \ref{nuv}(iii). Thus, we are done in case (a), and, in case (b), to finish the proof of \eqref{vala}, it remains to consider the case $i=2$. Consider, as in Lemma \ref{ulr}(ii), the dual Ulrich bundle $\F=\E^*(n+1-i_X)=\E^*(\frac{2u}{5})$.
Since $\Lambda^3 \E^* \cong (\Lambda^2 \E)(-u)$, we have, using Serre's duality,
$$H^2(\Lambda^3 \E^*)=H^2((\Lambda^2 \E)(-u))=H^2((\Lambda^2 \F^*)(\frac{4u}{5}-u))=H^{n-2}(\omega_X \otimes \Lambda^2 \F \otimes \L)$$
where $\L=\O_X(\frac{u}{5})$ is ample. Therefore $H^{n-2}(\omega_X \otimes \Lambda^2 \F \otimes \L)=0$ by \cite[Ex.~7.3.17]{laz2} and the conditions in (b) since $\F$ is globally generated by Lemma \ref{ulr}(i). This proves case (b).
\end{proof}

In the case of a hypersurface, the following result guarantees that $Z$ is connected.

\begin{lemma} 
\label{ipe}

Let $X \subset \P^{n+1}$ be a general smooth hypersurface of degree $d \ge 2$ and let $\E$ be a rank $r$ Ulrich bundle. Let $n, r, Z$ be as in Setup \ref{set} and assume that $Z \ne \emptyset$. If 
\begin{equation}
\label{dgr}
\binom{d+n+1-r}{n+1-r} \ge r(n+2-r)+1
\end{equation}
then $Z$ is smooth and irreducible.
\end{lemma}
\begin{proof}
Since $X$ is general of degree $d$ as above, $X$ does not contain linear subspaces of dimension $n+1-r$. Hence $\E$ is $(n-r)$-ample by \cite[Thm.~1]{lr2} and therefore $Z$ is smooth and irreducible by Lemma \ref{nonvuoto}(ii).
\end{proof}

\begin{remark}
\label{ipe2}
For later purposes we note that \eqref{dgr} holds for $n=8$ when $r=6, d \ge 4$ or when $r=7, d \ge 6$.
\end{remark}

\section{Invariants of hypersurfaces and their Ulrich bundles}
\label{inv}

In this section we will collect some invariants of hypersurfaces that will be used to prove our main theorem.

We start with a well-known, very useful, fact.

\begin{prop}
\label{lef}
Let $X \subset \P^{n+1}$ be a smooth irreducible hypersurface of dimension $n \ge 2$, degree $d$ with hyperplane section $H$. For $0 \le i \le n$ we have:
\begin{itemize}
\item[(i)] $H^{2i}(X,\Z) \cong \Z H^i$ for $i < \frac{n}{2}$.
\item[(ii)] $H^{2i}(X,\Z) \cong \Z \frac{1}{d}H^i$ for $i > \frac{n}{2}$.
\item[(iii)] If $n$ is even, $n \ge 3$, $X$ is very general and $d \ge 3$, then any algebraic class in $H^n(X,\Z)$ is of type $aH^{\frac{n}{2}}$ for some $a \in \Z$.
\end{itemize} 
\end{prop}
\begin{proof}
(i) follows by Lefschetz's hyperplane theorem, while (ii) follows by (i) and Poincar\'e's duality (see for example \cite[Ex.~1.2]{hu}). (iii) follows by Deligne's version of the Noether-Lefschetz's theorem (see for example \cite[Thm.~1.1]{sp}).
\end{proof}

We will use the following consequence of the above proposition.

\begin{cor}
\label{ce}
Let $n \ge 3$, let $X \subset \P^{n+1}$ be a smooth irreducible hypersurface of degree $d$ and let $\E$ be a globally generated vector bundle on $X$. For $i = \frac{n}{2}$, assume that one of the following holds:
\begin{itemize}
\item[(a)] $X$ is hyperplane section of a smooth hypersurface $X' \subset \P^{n+2}$ and $\E = \E'_{|X}$, where $\E'$ is a vector bundle on $X'$, or
\item[(b)] $X$ is very general and $d \ge 3$. 
\end{itemize}
Then, for all $1 \le i \le \frac{n}{2}$ (respectively $\frac{n}{2} < i \le n$), there exist $e_i \in \Z$ (resp. $e_i \in \Q$) such that $c_i(\E) = e_iH^i$ on $H^{2i}(X, \Z)$ (resp. on $H^{2i}(X, \Q)$). 
\end{cor}
\begin{proof}
In fact, if $i \ne \frac{n}{2}$, the conclusion follows by Proposition \ref{lef}(i)-(ii). Now suppose that $n$ is even and $i=\frac{n}{2}$. Under hypothesis (a), we have that $c_i(\E') = e_i(H')^i$ on $X'$, for some $e_i \in \Z$ and $H' \in |\O_{X'}(1)|$ by Proposition \ref{lef}(i). Hence also $c_i(\E) = c_i(\E'_{|X})=e_iH^i$. Under hypothesis (b), using Proposition \ref{lef}(iii), all we need to observe is that $c_i(\E)$ is algebraic. Even though the latter fact is well-known, we add a proof for completeness' sake. We can assume that $c_i(\E) \ne 0$, so that $r:= \rk \E \ge i$. Let $\varphi : \O_X^{\oplus (r-i+1)} \to \E$ be a general morphism and consider its degeneracy locus $D_{r-i}(\varphi)$. Observe that $D_{r-i}(\varphi) \ne \emptyset$, for otherwise we have an exact sequence
$$0 \to \O_X^{\oplus (r-i+1)}  \to \E \to \F \to 0$$
where also $\F$ is a vector bundle of rank $i-1$. But then $c_i(\E)=c_i(\F)=0$, a contradiction. Therefore $c_i(\E)=[D_{r-i}(\varphi)]$ is algebraic and we are done.
\end{proof}

Next, we compute the Chern classes of hypersurfaces.

\begin{lemma}
\label{xn}
Let $X \subset \P^{n+1}$ be a smooth irreducible hypersurface of degree $d$ with hyperplane section $H$. Then we have $c_i(X)=[\sum\limits_{k=0}^i(-1)^{i-k}\binom{n+2}{k}d^{i-k}]H^i$ for all $1 \le i \le n$.
\end{lemma}
\begin{proof}
From the Euler sequence
$$0 \to \O_X \to \O_X(1)^{\oplus (n+2)} \to {T_{\P^{n+1}}}_{|X} \to 0$$
we find that $c_i({T_{\P^{n+1}}}_{|X})=\binom{n+2}{i}H^i$ and the normal bundle sequence
$$0 \to T_X \to {T_{\P^{n+1}}}_{|X} \to \O_X(d) \to 0$$
gives, for $i \ge 1$, that $c_i(X)=\binom{n+2}{i}H^i-dHc_{i-1}(X)$. Now the statement follows by induction on $i$.
\end{proof}

We now compute Chern classes of Ulrich bundles on hypersurfaces.

\begin{lemma}
\label{xne}
Let $n \ge 3$ and let $X \subset \P^{n+1}$ be a smooth irreducible hypersurface of degree $d$ with hyperplane section $H$. Let $\E$ be a rank $r$ Ulrich bundle on $X$. If $i \le \frac{n}{2}$ consider $c_i(\E) \in H^{2i}(X, \Z)$, if $i > \frac{n}{2}$ consider $c_i(\E) \in H^{2i}(X, \Q)$; if $i=\frac{n}{2}$ assume in addition that either (a) or (b) of Corollary \ref{ce} holds. Then:
\begin{itemize}
\item[(1)] $c_1(\E)=\frac{r}{2}(d-1)H$.
\item[(2)] $c_2(\E)=\frac{r}{24}(d-1)(3rd-2d-3r+4)H^2$.
\item[(3)] $c_3(\E)=\frac{r}{48}(r-2)(d-1)^2(dr-r+2)H^3$.
\item[(4)] $\begin{aligned}[t] 
c_4(\E)=\frac{r}{5760}(d-1)&[(15r^3-60r^2+20r+48)d^3-(45r^3-240r^2+340r-48)d^2+\\
& +(45r^3-300r^2+640r-432)d-15r^3+120r^2-320r+288]H^4
\end{aligned}$
\item[(5)] If $r=5$, $c_5(\E)=\frac{1}{2304}(d-1)^2(5d-1)(23d^2-54d+19)H^5$.
\item[(6)] If $r=6$, $c_5(\E)=\frac{1}{40}(d-1)^2(2d-1)(2d-3)(3d-1)H^5$.
\item[(7)] If $r=6$, $c_6(\E)=\frac{1}{1680}(d-1)(2d-1)(3d-1)(6d-1)(5-3d+2d^2)H^6$.
\item[(8)] If $r=7$, $c_5(\E)=\frac{7}{3840}(d-1)^2(7d-3)(59-150d+79d^2)$
\item[(9)] If $r=7$, $\begin{aligned}[t]
c_6(\E)=\frac{1}{414720}(d-1)(& -13837+119975d-375310d^2+524330d^3-330853d^4\\
& +87215d^5)H^6.
\end{aligned}$
\item[(10)] If $r=7$, $c_7(\E)=\frac{1}{829440}(d-1)^2(7d-1)(913-5620d+10170d^2-6380d^3+2837d^4)H^7$.
\end{itemize}
\end{lemma}
\begin{proof}
We use Lemma \ref{xn} and Corollary \ref{ce}. The formulas (1)-(4) follow by using Lemma \ref{xn} and Lemma \ref{ulr}(iv), (vii), (ix), (x) (see \cite[Out(29), Out(35), Out(41)]{mcc}). Formula (5) (respectively (6)-(7), resp. (8)-(10)) follows by restricting to $X_5$ (resp. to $X_6$; resp. to $X_7$), using Lemma \ref{ulr}(xi)-(xiii) (see \cite[Out(43)]{mc5}, \cite[Out(124), Out(130)]{mc6}, \cite[Out(87), Out(93), Out(99)]{mc7}).
\end{proof}

\section{Proof of Theorem \ref{main}}
\label{pf}

We will prove the theorem by using the degeneracy locus introduced in section \ref{dege}.

In order to simplify statements we will use the following 

\begin{setup}
\label{set2}

\hskip 3cm

\begin{itemize} 
\item $X \subset \P^{n+1}$ is a smooth irreducible hypersurface of degree $d \ge 3$ with hyperplane section $H$.
\item $\E$ is a rank $r$ Ulrich bundle on $X$.
\item $V \subset H^0(\E)$ is a general subspace of dimension $2$, giving rise to
$\varphi : V \otimes \O_X \to \E$.
\item $Z =D_1(\varphi)$ is the corresponding degeneracy locus, $H_Z=H_{|Z}$. 
\end{itemize} 
\end{setup}

The next goal is to compute the necessary invariants of $Z$.

First, we do it in dimension $6$.

\begin{lemma} 
\label{x6z}
Notation as in Setup \ref{set2} with $n=6$. 

If $r=4$, the following hold:
\begin{itemize}
\item[(i)] $Z$ is a smooth irreducible threefold.
\item[(ii)] $\deg Z=\frac{d}{3}(d-1)^2(2d-1)$.
\item[(iii)] $2c_2(Z)=-\frac{4}{3}(2d-5)(5d-19)H_Z^2+(8d-22)K_ZH_Z$.
\end{itemize}
Moreover we have a resolution
\begin{equation}
\label{res4}
0 \to \O_X^{\oplus 3} \to \E^{\oplus 2} \to \Lambda^2 \E \to \I_{Z/X}(2d-2) \to 0
\end{equation}
If $r=5$, the following hold:
\begin{itemize}
\item[(iv)] $Z$ is a smooth irreducible surface.
\item[(v)] $\deg Z=\frac{d}{1152}(d-1)(-187+893d-1277d^2+523d^3)$.
\item[(vi)] $K_Z^2=(7d-21)K_ZH_Z-\frac{1}{4}(7d-21)^2\deg Z$.
\item[(vii)] $3c_2(Z)=-\frac{1}{8}(195d^2-1132d+1609)\deg Z+(13d-34)K_ZH_Z$.
\end{itemize}
Moreover we have a resolution
\begin{equation}
\label{res5}
0 \to \O_X^{\oplus 4} \to \E^{\oplus 3} \to (\Lambda^2 \E)^{\oplus 2} \to \Lambda^3 \E \to \I_{Z/X}(\frac{5}{2}(d-1)) \to 0
\end{equation}
\end{lemma}
\begin{proof}
Since $(X_3, \E_{|X_3})$ satisfies (a) of Corollary \ref{ce}, we get by Lemma \ref{xne}(3), that 
\begin{equation}
\label{c_3}
c_3(\E)H^3=c_3(\E_{|X_3})=\frac{d}{3}(d-1)^2(2d-1) 
\end{equation}
so that, in particular, $c_3(\E) \ne 0$. It follows by Lemma \ref{nonvuoto} that $Z \ne \emptyset$, hence Lemma \ref{nldg} applies. Therefore $[Z]=c_3(\E)$, $Z$ is a smooth threefold and (ii) follows by \eqref{c_3}. We get \eqref{res4} by \eqref{ea1} and Lemma \ref{xne}(1). Moreover, from \eqref{ch2}, we get (iii) using Lemmas \ref{xn} and \ref{xne}(1)-(2). Also, $Z$ is irreducible by Lemma \ref{duecasi}(a), proving (i). Next, (iv)-(vii) and \eqref{res5} are proved in the same way using the same lemmas.
\end{proof}

Next, we do it in dimension $8$.

\begin{lemma} 
\label{x8z}
Notation as in Setup \ref{set2} with $n=8$. 

If $r=6$, the following hold:
\begin{itemize}
\item[(i)] $Z$ is a smooth threefold.
\item[(ii)] $\deg Z=\frac{d}{40}(d-1)^2(2d-1)(2d-3)(3d-1)$.
\item[(iii)] $4c_2(Z)=-(393-253d+40d^2)H_Z^2+(19d-55)K_ZH_Z$.
\end{itemize}
Moreover we have a resolution
\begin{equation}
\label{res6}
0 \to \O_X^{\oplus 5} \to \E^{\oplus 4} \to (\Lambda^2 \E)^{\oplus 3} \to (\Lambda^3 \E)^{\oplus 2} \to \Lambda^4 \E \to \I_{Z/X}(3d-3) \to 0
\end{equation}
If $r=7$, the following hold:
\begin{itemize} 
\item[(iv)] $Z$ is a smooth surface.
\item[(v)] $\deg Z=\frac{d}{414720}(d-1)(-13837+119975d-375310d^2+524330d^3-330853d^4+87215d^5).$
\item[(vi)] $K_Z^2=(9d-27)K_ZH_Z-\frac{1}{4}(9d-27)^2\deg Z$.
\item[(vii)] $5c_2(Z)=-\frac{1}{24}(12529-8592d+1463d^2)\deg Z+(26d-71)K_ZH_Z$.
\end{itemize}
Moreover we have a resolution
\begin{equation}
\label{res7}
0 \to \O_X^{\oplus 6} \to \E^{\oplus 5} \to (\Lambda^2 \E)^{\oplus 4} \to (\Lambda^3 \E)^{\oplus 3} \to (\Lambda^4 \E)^{\oplus 2} \to \Lambda^5 \E \to \I_{Z/X}(\frac{7}{2}(d-1)) \to 0
\end{equation}
\end{lemma}
\begin{proof}
Similar to the proof of Lemma \ref{x6z}.
\end{proof}

\begin{remark}
If $X$ is general and $r=6, d \ge 4$ or  $r=7, d \ge 6$, we have, by Lemma \ref{ipe} and Remark \ref{ipe2}, that $Z$ in Lemma \ref{x8z} is irreducible. However this is not needed for our purposes, see Remark \ref{funz}.
\end{remark}

We are now ready for the proof of the main theorem.

\renewcommand{\proofname}{Proof of Theorem \ref{main}}
\begin{proof}
Let $X \subset \P^{n+1}$ be a smooth hypersurface of degree $d \ge 3$ with hyperplane section $H$ and let $\E$ be a rank $r$ Ulrich bundle on $X$. Under the hypotheses of the theorem, it follows by \cite[Thm.~2]{lr3} that $\uc(X) \ge 4$, hence we need to consider the cases $r=4, 5, 6, 7$.

Let $V \subset H^0(\E)$ be a general subspace of dimension $2$, giving rise to $\varphi : V \otimes \O_X \to \E$ and let 
$$Z =D_1(\varphi)$$
be the corresponding degeneracy locus, with $H_Z=H_{|Z}$.

The plan is to compute the Euler characteristic of $Z$ in two different ways and get a contradiction. In order to do this, we distinguish two cases.

When $Z$ is a surface, we will use the formula below, that follows by Riemann-Roch,
\begin{equation}
\label{hk}
K_ZH_Z=-2\chi(\O_Z(1))+2\chi(\O_Z)+\deg(Z)
\end{equation}
together with \eqref{res5}, Lemma \ref{x6z}(v)-(vii) and Lemma \ref{x8z}(v)-(vii).

When $Z$ is a threefold, we will use the formulas below, that follow again by Riemann-Roch,
\begin{equation}
\label{kh2}
K_Z H_Z^2=4\chi(\O_Z(1))-2\chi(\O_Z(2))-2\chi(\O_Z)+2\deg(Z)
\end{equation}
and
\begin{equation}
\label{kh2+hc_2}
K_Z^2 H_Z+H_Z c_2(Z)=12\chi(\O_Z(1))-12\chi(\O_Z)-2\deg(Z)+3K_Z H_Z^2.
\end{equation}
together with \eqref{res4}, \eqref{res6}, Lemma \ref{x6z}(ii)-(iii) and Lemma \ref{x8z}(ii)-(iii).

Now suppose that $r=4$.

If $n \ge 7$ we have that $\E_{|X_6}$ is a rank $4$ Ulrich bundle on $X_6 \subset \P^7$ by Lemma \ref{ulr}(v), therefore condition (a) of Corollary \ref{ce} is satisfied for $(X_6,\E_{|X_6})$. Hence, in order to show that there is no rank $4$ Ulrich bundle on $X$, we can assume that $n=6$ and that $X$ satisfies either (a) or (b) of Corollary \ref{ce}.

Let $m \in \Z$. First, from
$$0 \to \O_{\P^7}(m-d) \to \O_{\P^7}(m) \to \O_X(m) \to 0$$
we have that
\begin{equation}
\label{chi6}
\chi(\O_X(m))=\chi(\O_{\P^7}(m))-\chi(\O_{\P^7}(m-d))=\binom{m+7}{7}-\binom{m-d+7}{7}.
\end{equation}
From \eqref{res4} we find
$$0 \to \O_X^{\oplus 3}(m-2d+2) \to \E^{\oplus 2}(m-2d+2) \to (\Lambda^2 \E)(m-2d+2) \to \I_{Z/X}(m) \to 0.$$
Using the above, \eqref{chi6} and Lemma \ref{ulr}(viii), we get
\begin{equation}
\label{chiz4}
\begin{aligned}[t] 
\chi(\O_Z(m)) & = \chi(\O_X(m))-\chi(\I_{Z/X}(m))=\\
& = \chi(\O_X(m))-\chi((\Lambda^2 \E)(m-2d+2))+2\chi(\E(m-2d+2))-3\chi(\O_X(m-2d+2)).\\
& = \binom{m+7}{7}-\binom{m-d+7}{7}-\chi((\Lambda^2 \E)(m-2d+2))+8d\binom{m-2d+8}{6}\\
& -3\binom{m-2d+9}{7}+3\binom{m-3d+9}{7}.
\end{aligned}
\end{equation}
Using Lemmas \ref{xn}, \ref{xne}, the expression of $\chi((\Lambda^2 \E)(m-2d+2))$ is computed in the Appendix, Lemma \ref{suz4}(1). Setting $m=0$ in \eqref{chiz4}, we find (see \cite[Out(103)]{mc4}),
\begin{equation}
\label{chi60}
\chi(\O_Z)=-\frac{d}{340200}(d-1)(2d-1)(2303699-4840923d+3320849d^2-947157d^3+97472d^4).
\end{equation}
Similarly, setting $m=1, 2$, we get (see \cite[Out(105), Out(107)]{mc4}),
$$\chi(\O_Z(1))=-\frac{d}{340200}(d-1)(2d-1)(4034939-7679703d+4543679d^2-1107807d^3+97472d^4)$$
and
$$\chi(\O_Z(2))=-\frac{d}{340200}(d-1)(2d-1)(6454139-11403003d+5951729d^2-1268457d^3+97472d^4).$$
Using the above, Lemma \ref{x6z}(ii), \eqref{kh2} and \eqref{kh2+hc_2} we have (see \cite[Out(110), Out(112)]{mc4}),
$$K_ZH_Z^2= \frac{d}{45}(d-1)(2d-1)(152-204d+49d^2)$$
and
$$K_Z^2 H_Z+H_Z c_2(Z)=\frac{d}{15}(d-1)(2d-1)(-754+1288d-598d^2+85d^3).$$
Next, using the above and Lemma \ref{x6z}(ii)-(iii)(see \cite[Out(114), Out(116)]{mc4}), we have
$$H_Z c_2(Z)= \frac{d}{45}(d-1)(2d-1)(-722+1272d-625d^2+96d^3)$$
and
$$K_Z^2H_Z=\frac{d}{45}(d-1)(2d-1)(3d-10)(154-213d+53d^2).$$
Using the above and Lemma \ref{x6z}(iii) we find (see \cite[Out(118)]{mc4}),
$$K_Z c_2(Z)=\frac{d}{135}(d-1)(2d-1)(21940-46104d+31627d^2-9021d^3+928d^4).$$
Since $\chi(\O_Z)=-\frac{1}{24}K_Z c_2(Z)$, equating with \eqref{chi60} (see \cite[Out(120)]{mc4}) gives the contradiction
$$(-1+d)d(1+d)(-1+2d)(1+2d)(-1+4d)(1+4d)=0.$$
Next, suppose that $r=5$.

Arguing as in the beginning of the case $r=4$, we can assume that $n=6$ and that $X$ satisfies either (a) or (b) of Corollary \ref{ce}.

From \eqref{res5}, setting $u=\frac{5}{2}(d-1)$, we find
$$0 \to \O_X(m-u)^{\oplus 4} \to \E(m-u)^{\oplus 3} \to (\Lambda^2 \E)(m-u)^{\oplus 2} \to (\Lambda^3 \E)(m-u) \to \I_{Z/X}(m) \to 0.$$
Using the above, \eqref{chi6} and Lemma \ref{ulr}(viii) one gets
\begin{equation}
\label{chiz5}
\begin{aligned}[t] 
\chi(\O_Z(m)) & = \binom{m+7}{7}-\binom{m-d+7}{7}-\chi((\Lambda^3 \E)(m-u))+2\chi((\Lambda^2 \E)(m-u))\\& -15d\binom{m-u+6}{6}+4\binom{m-u+7}{7}-4\binom{m-u-d+7}{7}.
\end{aligned}
\end{equation}
Using Lemmas \ref{xn}, \ref{xne}, the expressions of $\chi((\Lambda^2 \E)(m-u))$ and $\chi((\Lambda^3 \E)(m-u))$ are computed in the Appendix, Lemma \ref{suz5}(1)-(2). Setting $m=0$ in \eqref{chiz5}, we find (see \cite[Out(53)]{mc5}),
\begin{equation}
\label{chi601}
\chi(\O_Z)=\frac{d}{1548288}(d-1)(-3500495+19507441d-37476458d^2+30435862d^3-10691399d^4+1349497d^5).
\end{equation}
Similarly one gets (see \cite[Out(55)]{mc5})
$$\chi(\O_Z(1))=\frac{d}{1548288}(d-1)(-4964783+27017713d-49890986d^2+37892374d^3-12037415d^4+1349497 d^5).$$
Using \eqref{hk} and Lemma \ref{x6z}(v), we get (see \cite[Out(57)]{mc5}),
$$K_Z H_Z=\frac{d}{1152}(d-1)(1992-10283d+17197d^2-10573d^3+2003d^4).$$
Now, using the above and Lemma \ref{x6z}(v)-(vii)(see \cite[Out(59), Out(61)]{mc5}) we find
$$K_Z^2=\frac{7d}{4608}(d-3)(d-1)(4041-21070d+35720d^2-22370d^3+4351d^4)$$
and
$$c_2(Z)=\frac{d}{27648}(d-1)(-240941+1355623d-2644982d^2+2203138d^3-803357d^4+106327d^5).$$
Since $\chi(\O_Z)=\frac{1}{12}(K_Z^2+c_2(Z))$, equating with \eqref{chi601} (see \cite[Out(63)]{mc5}), one gets the contradiction 
$$(-1+d)d(1+d)(-1+5d)(1+5d)(-13+ 61d^2)=0.$$
Now, suppose that $r=6$.

Arguing as in the beginning of the case $r=4$, we can assume that $n=8$ and that $X$ satisfies either (a) or (b) of Corollary \ref{ce}.

Let $m \in \Z$. First, from
$$0 \to \O_{\P^9}(m-d) \to \O_{\P^9}(m) \to \O_X(m) \to 0$$
we have that
\begin{equation}
\label{chi8}
\chi(\O_X(m))=\chi(\O_{\P^9}(m))-\chi(\O_{\P^9}(m-d))=\binom{m+9}{9}-\binom{m-d+9}{9}.
\end{equation}
From \eqref{res6}, we find
$$\begin{aligned}[t] 
0 \to \O_X^{\oplus 5}(m-3d+3) \to \E^{\oplus 4}(m-3d+3) \to (\Lambda^2 \E)^{\oplus 3}(& m-3d+3) \to (\Lambda^3 \E)^{\oplus 2}(m-3d+3) \to\\
& \to (\Lambda^4 \E)(m-3d+3) \to \I_{Z/X}(m) \to 0.
\end{aligned}$$
Using the above, \eqref{chi8} and Lemma \ref{ulr}(viii), we get
\begin{equation}
\label{chiz6}
\begin{aligned}[t] 
\chi(\O_Z(m)) & = \binom{m+9}{9}-\binom{m-d+9}{9}-\chi((\Lambda^4 \E)(m-3d+3))+2\chi((\Lambda^3 \E)(m-3d+3))\\ & -3\chi((\Lambda^2 \E)(m-3d+3))-24d\binom{m-3d+11}{8}+5\binom{m-3d+12}{9}-5\binom{m-4d+12}{9}.
\end{aligned}
\end{equation}
Using Lemmas \ref{xn}, \ref{xne}, the expressions of $\chi((\Lambda^i \E)(m-3d+3)), 2 \le i \le 4$ are computed in the Appendix, Lemma \ref{suz6}. Setting $m=0$ in \eqref{chiz6}, we find (see \cite[Out(144)]{mc6}),
\begin{equation}
\label{chi80}
\begin{aligned}[t] 
\chi(\O_Z)=-\frac{d}{84672000}(d-1)(2d-1)(3d-1)(& -287792399+809751606d-812826025d^2+397479390d^3\\
& -96129996d^4+9172584d^5).
\end{aligned}
\end{equation}
Similarly, setting $m=1, 2$, we get (see \cite[Out(147), Out(150)]{mc6}),
$$\begin{aligned}[t] 
\chi(\O_Z(1))=-\frac{d}{84672000}(d-1)(2d-1)(3d-1)(& -445115999+1180443606d-1099476025d^2\\
& +492231390d^3-107520396 d^4+9172584d^5)
\end{aligned}$$
and
$$\begin{aligned}[t] 
\chi(\O_Z(2))=-\frac{d}{84672000}(d-1)(2d-1)(3d-1)(&-650571599+1646542806d-1441062025d^2\\
& +596660190d^3-118910796d^4+9172584 d^5).
\end{aligned}$$
Using the above, Lemma \ref{x8z}(ii), \eqref{kh2} and \eqref{kh2+hc_2} we have (see \cite[Out(153), Out(155)]{mc6}),
$$K_Z H_Z^2= \frac{d}{840}(d-1)(2d-1)(3d-1)(-829+1683d-1006d^2+192 d^3)$$
and
$$K_Z^2 H_Z+H_Z c_2(Z)=\frac{d}{280}(d-1)(2d-1)(3d-1)(5372-12957d+10341d^2-3568d^3+452d^4).$$
Next, using the above and Lemma \ref{x8z}(ii)-(iii)(see \cite[Out(158), Out(161)]{mc6}), we have
$$H_Z c_2(Z)= \frac{d}{840}(d-1)(2d-1)(3d-1)(5209-12778d+10429d^2-3712d^3+492d^4)$$
and
$$K_Z^2H_Z=\frac{d}{840}(d-1)(2d-1)(3d-1)(4d-13)(-839+1749d-1046d^2+216d^3).$$
Using the above and Lemma \ref{x8z}(iii) we find (see \cite[Out(164)]{mc6}),
$$K_Z c_2(Z)=\frac{d}{420}(d-1)(2d-1)(3d-1)(-34261+96399d-96765d^2+47319d^3-11444d^4+1092d^5).$$
Since $\chi(\O_Z)=-\frac{1}{24}K_Z c_2(Z)$, equating with \eqref{chi80} (see \cite[Out(166)]{mc6}) gives the contradiction
$$(-1+d)d(1+d)(-1+2d)(1+2d)(-1+3d)(1+3d)(-1+6d)(1+6d)=0.$$
Finally, suppose that $r=7$.

Arguing as in the beginning of the case $r=4$, we can assume that $n=8$ and that $X$ satisfies either (a) or (b) of Corollary \ref{ce}.

From \eqref{res7}, setting $u=\frac{7}{2}(d-1)$, we find
$$0 \to \O_X(m-u)^{\oplus 6} \to \E(m-u)^{\oplus 5} \to (\Lambda^2 \E)(m-u)^{\oplus 4} \to (\Lambda^3 \E)(m-u)^{\oplus 3} \to $$
$$ \hskip 4.5cm \to (\Lambda^4 \E)(m-u)^{\oplus 2} \to (\Lambda^5 \E)(m-u)\to \I_{Z/X}(m) \to 0.$$
Using the above, \eqref{chi8} and Lemma \ref{ulr}(viii) one gets
\begin{equation}
\label{chiz7}
\begin{aligned}[t] 
\chi(\O_Z(m)) & = \binom{m+9}{9}-\binom{m-d+9}{9}-\chi((\Lambda^5 \E)(m-u))+2\chi((\Lambda^4 \E)(m-u))-3\chi((\Lambda^3 \E)(m-u))\\
& +4\chi((\Lambda^2 \E)(m-u)) -35d\binom{m-u+8}{8}+6\binom{m-u+9}{9}-6\binom{m-u-d+9}{9}.
\end{aligned}
\end{equation}
Using Lemmas \ref{xn}, \ref{xne}, the expressions of $\chi((\Lambda^i \E)(m-u)), 2 \le i \le 5$ are computed in the Appendix, Lemma \ref{suz7}. Setting $m=0$ in \eqref{chiz7}, we find (see \cite[Out(115)]{mc7}),
\begin{equation}
\label{chi801}
\begin{aligned}[t] 
\chi(\O_Z)=\frac{d(d-1)}{28665446400}(& -22024437079+208787633321d-751494758379d^2+1321535623701d^3\\
& -1237566062181d^4+646601246619d^5-177940027481d^6\\
& +19863510439d^7).
\end{aligned}
\end{equation}
Similarly one gets (see \cite[Out(117)]{mc7})
$$\begin{aligned}[t]
\chi(\O_Z(1))=\frac{d(d-1)}{28665446400}(& -29037317719+272178069161d-963544031979d^2+1653635796501d^3\\
& -1495712707941d^4+747580244379d^5-193219521881d^6+19863510439d^7).
\end{aligned}$$
Using \eqref{hk} and Lemma \ref{x8z}(v), we get (see \cite[Out(121)]{mc7}),
$$K_Z H_Z=\frac{d(d-1)}{414720}(189082-1714239d+5760375d^2-9085050d^3+7138668d^4-2834631d^5+442115d^6).$$
Now, using the above and Lemma \ref{x8z}(v)-(vii)(see \cite[Out(123), Out(125)]{mc7}) we find
$$K_Z^2=\frac{d(d-1)}{184320}(d-3)(382729-3493098d+11828355d^2-18805500d^3+14902671d^4-6006042d^5+983525d^6)$$
and
$$\begin{aligned}[t]
c_2(Z)=\frac{d(d-1)}{9953280}(& -29766391+283399229d-1026407283d^2+1821176337d^3-1726796469d^4\\
& +916447911d^5-257756897d^6+29656843d^7).
\end{aligned}$$
Since $\chi(\O_Z)=\frac{1}{12}(K_Z^2+c_2(Z))$, equating with \eqref{chi801} (see \cite[Out(127)]{mc7}), one gets the contradiction 
$$d(-1+d)(1+d)(-1+7d)(1+7d)(281-4210d^2+12569d^4)=0.$$
This concludes the proof of the theorem.
\end{proof}
\renewcommand{\proofname}{Proof}

\begin{remark}
\label{funz}
In the above calculations, we used the formulas $\chi(\O_Z)=\frac{1}{12}(c_1(Z)^2+c_2(Z))$ for a smooth surface $Z$ and
$\chi(\O_Z)=\frac{1}{24}c_1(Z)c_2(Z)$ for a smooth threefold $Z$. We point out that these formulas make sense even though $Z$ might be disconnected (and so do the formulas \eqref{hk}, \eqref{kh2}, \eqref{kh2+hc_2} and the other formulas used).
\end{remark}
As a matter of fact, let $Z = Z_1 \sqcup \ldots \sqcup Z_s$ be the decomposition into connected components and let $j_k : Z_k \hookrightarrow Z$ be the inclusion, for $1 \le k \le s$. 

We have that $\O_Z \cong \O_{Z_1} \oplus \ldots \oplus \O_{Z_s}$ and, by \cite[Rmk.~II.8.9.2]{ha}, $T_Z \cong T_{Z_1} \oplus \ldots \oplus T_{Z_s}$. Therefore $\chi(\O_Z)=\sum_{i=1}^s \chi(\O_{Z_i})$. So, for example if $Z$ is a smooth surface, we get 
\begin{equation}
\label{chirr}
\chi(\O_Z)=\frac{1}{12}\sum\limits_{i=1}^s (c_1(Z_i)^2+c_2(Z_i))=\frac{1}{12}(\sum\limits_{i=1}^s c_1(Z_i)^2+\sum\limits_{i=1}^s c_2(Z_i)).
\end{equation}
On the other hand, consider for $p=1, 2$, the isomorphism $H^{2p}(Z,\Z) \cong H^{2p}(Z_1,\Z) \oplus \ldots \oplus H^{2p}(Z_s,\Z)$ given by $j_1^* \oplus \ldots \oplus j_s^*$. Then
\begin{equation}
\label{che}
\begin{aligned}[t]
c_p(Z)& = c_p(T_Z)= j_1^*c_p(T_Z)+ \ldots + j_s^* c_p(T_Z)=c_p(j_1^*T_Z) + \ldots + c_p(j_s^*T_Z)=c_p(T_{Z_1}) + \ldots + c_p(T_{Z_s})=\\
& = c_p(Z_1) + \ldots + c_p(Z_s).
\end{aligned}
\end{equation}
Since, by definition of cup product, we have that $\alpha \beta = 0$ if $\alpha \in H^{2p}(Z_k,\Z), \beta \in H^{2q}(Z_h,\Z)$ with $k \ne h$, we deduce that 
$$c_1(Z)^2=c_1(Z_1)^2 + \ldots + c_1(Z_s)^2$$
and therefore, by \eqref{chirr} and \eqref{che},
$$\chi(\O_Z)=\frac{1}{12}(\sum\limits_{i=1}^s c_1(Z_i)^2+\sum\limits_{i=1}^s c_2(Z_i))=\frac{1}{12}(c_1(Z)^2+c_2(Z)).$$
A similar calculation can be done when $Z$ is a smooth threefold or for the other formulas.

\section*{Acknowledgement}
We thank Prof. Marco Pedicini (Roma Tre University) for the help in programming the (long) case $r=7$ \cite{mc7}.

\eject

\appendix

\section{Chern classes of exterior powers}
\label{app1}

\begin{lemma} 
\label{w4}
Let $\F$ be a rank $4$ vector bundle on a smooth variety $X$. Then:
\begin{enumerate}
\item $c_1(\Lambda^2 \F)=3c_1(\F)$.
\item $c_2(\Lambda^2 \F)=3c_1(\F)^2+2c_2(\F)$.
\item $c_3(\Lambda^2 \F)=c_1(\F)^3+4c_1(\F)c_2(\F)$.
\item $c_4(\Lambda^2 \F)=2c_1(\F)^2 c_2(\F)+c_2(\F)^2+c_1(\F)c_3(\F)-4c_4(\F)$.
\item $c_5(\Lambda^2 \F)=c_1(\F)c_2(\F)^2+c_1(\F)^2 c_3(\F)-4c_1(\F)c_4(\F)$.
\item $c_6(\Lambda^2 \F)=c_1(\F)c_2(\F)c_3(\F)-c_3(\F)^2-c_1(\F)^2 c_4(\F)$.
\end{enumerate}
\end{lemma}
\begin{proof}
See Out(80), Out(82), Out(84), Out(86), Out(88), Out(90) in \cite{mc4}. 
\end{proof}

\begin{lemma} 
\label{w5}
Let $\F$ be a rank $5$ vector bundle on a smooth variety $X$. Then:
\begin{enumerate}
\item $c_1(\Lambda^2 \F)=4c_1(\F)$.
\item $c_2(\Lambda^2 \F)=6c_1(\F)^2+3c_2(\F)$.
\item $c_3(\Lambda^2 \F)=4c_1(\F)^3+9c_1(\F)c_2(\F)+c_3(\F)$.
\item $c_4(\Lambda^2 \F)=c_1(\F)^4+9c_1(\F)^2 c_2(\F)+3c_2(\F)^2+4c_1(\F)c_3(\F)-3c_4(\F)$.
\item $c_5(\Lambda^2 \F)=3c_1(\F)^3 c_2(\F)+6c_1(\F)c_2(\F)^2+5c_1(\F)^2 c_3(\F)+2c_2(\F)c_3(\F)-5c_1(\F)c_4(\F)-11c_5(\F)$.
\item $\begin{aligned}[t]
c_6(\Lambda^2 \F)=& \ 3c_1(\F)^2 c_2(\F)^2+c_2(\F)^3+2c_1(\F)^3 c_3(\F)+6c_1(\F)c_2(\F)c_3(\F)-c_3(\F)^2\\
& -2c_1(\F)^2 c_4(\F)-2c_2(\F)c_4(\F)-22c_1(\F)c_5(\F).
\end{aligned}$
\end{enumerate}
\end{lemma}
\begin{proof}
See Out(20), Out(22), Out(24), Out(26), Out(28), Out(30) in \cite{mc5}. 
\end{proof}

\begin{lemma}
\label{w6}
Let $\F$ be a rank $6$ vector bundle on a smooth variety $X$. Then:
\begin{enumerate}
\item $c_1(\Lambda^2 \F)=5c_1(\F)$.
\item $c_2(\Lambda^2 \F)=10c_1(\F)^2+4c_2(\F)$.
\item $c_3(\Lambda^2 \F)=10c_1(\F)^3+16c_1(\F)c_2(\F)+2c_3(\F)$.
\item $c_4(\Lambda^2 \F)=5c_1(\F)^4+24c_1(\F)^2 c_2(\F)+6c_2(\F)^2+9c_1(\F)c_3(\F)-2c_4(\F)$.
\item $\begin{aligned}[t] 
c_5(\Lambda^2 \F)=& \ c_1(\F)^5+16c_1(\F)^3 c_2(\F)+18c_1(\F)c_2(\F)^2+15c_1(\F)^2 c_3(\F)+6c_2(\F)c_3(\F) \\
& -4c_1(\F)c_4(\F)-10c_5(\F).
\end{aligned}$
\item $\begin{aligned}[t]
c_6(\Lambda^2 \F)=& \ 4c_1(\F)^4 c_2(\F)+18c_1(\F)^2 c_2(\F)^2+4c_2(\F)^3+11c_1(\F)^3 c_3(\F)+21c_1(\F)c_2(\F) c_3(\F)\\
& -c_1(\F)^2 c_4(\F)-2c_2(\F)c_4(\F)-29c_1(\F)c_5(\F)-26c_6(\F).
\end{aligned}$
\item $\begin{aligned}[t]
c_7(\Lambda^2 \F)=& \ 6c_1(\F)^3 c_2(\F)^2 +8c_1(\F)c_2(\F)^3+3c_1(\F)^4 c_3(\F)+24c_1(\F)^2 c_2(\F)c_3(\F)\\
& +6c_2(\F)^2 c_3(\F)+3c_1(\F)c_3(\F)^2+2c_1(\F)^3 c_4(\F)+2c_1(\F)c_2(\F)c_4(\F)\\
& -6c_3(\F)c_4(\F)-32c_1(\F)^2 c_5(\F)-12c_2(\F)c_5(\F)-78c_1(\F)c_6(\F).
\end{aligned}$ 
\item $\begin{aligned}[t]
c_8(\Lambda^2 \F)=& \ 4c_1(\F)^2 c_2(\F)^3+c_2(\F)^4+9c_1(\F)^3 c_2(\F)c_3(\F)+15c_1(\F)c_2(\F)^2 c_3(\F)\\
& +6c_1(\F)^2 c_3(\F)^2+c_1(\F)^4 c_4(\F)+8c_1(\F)^2 c_2(\F)c_4(\F)+2c_2(\F)^2 c_4(\F)\\
& -8c_1(\F)c_3(\F)c_4(\F)-7c_4(\F)^2-16c_1(\F)^3 c_5(\F)-26c_1(\F)c_2(\F)c_5(\F)\\
& -3c_3(\F)c_5(\F)-94c_1(\F)^2 c_6(\F)-24c_2(\F)c_6(\F).
\end{aligned}$ 
\item $c_1(\Lambda^3 \F)=10 c_1(\F)$.
\item $c_2(\Lambda^3 \F)=45c_1(\F)^2+6c_2(\F)$.
\item $c_3(\Lambda^3 \F)=120c_1(\F)^3+54c_1(\F)c_2(\F)$.
\item $c_4(\Lambda^3 \F)=210c_1(\F)^4+216c_1(\F)^2 c_2(\F)+15c_2(\F)^2+3c_1(\F)c_3(\F)-6c_4(\F)$.
\item $c_5(\Lambda^3 \F)=252c_1(\F)^5+504c_1(\F)^3 c_2(\F)+120c_1(\F)c_2(\F)^2+24c_1(\F)^2 c_3(\F)-48c_1(\F)c_4(\F)$.
\item $\begin{aligned}[t]
c_6(\Lambda^3 \F)=& \ 210c_1(\F)^6+756c_1(\F)^4 c_2(\F)+420c_1(\F)^2 c_2(\F)^2+20c_2(\F)^3+84c_1(\F)^3 c_3(\F)\\
& +15c_1(\F)c_2(\F)c_3(\F)-3c_3(\F)^2-169c_1(\F)^2 c_4(\F)-22c_2(\F)c_4(\F)\\
& -11c_1(\F)c_5(\F)+66c_6(\F).
\end{aligned}$
\item $\begin{aligned}[t]
c_7(\Lambda^3 \F)=& \ 120c_1(\F)^7+756c_1(\F)^5 c_2(\F)+840c_1(\F)^3 c_2(\F)^2+140c_1(\F)c_2(\F)^3+168c_1(\F)^4 c_3(\F)\\
& +105c_1(\F)^2 c_2(\F)c_3(\F)-21c_1(\F)c_3(\F)^2-343c_1(\F)^3 c_4(\F)-154c_1(\F)c_2(\F)c_4(\F)\\
& -77c_1(\F)^2 c_5(\F)+462c_1(\F)c_6(\F).
\end{aligned}$
\item $\begin{aligned}[t]
c_8(\Lambda^3 \F)=& \ 45c_1(\F)^8+504c_1(\F)^6 c_2(\F)+1050c_1(\F)^4 c_2(\F)^2+420c_1(\F)^2 c_2(\F)^3+15 c_2(\F)^4\\
& +210 c_1(\F)^5 c_3(\F)+315c_1(\F)^3 c_2(\F)c_3(\F)+30c_1(\F)c_2(\F)^2 c_3(\F)-60c_1(\F)^2 c_3(\F)^2\\
& -12c_2(\F)c_3(\F)^2-441c_1(\F)^4 c_4(\F)-465c_1(\F)^2 c_2(\F)c_4(\F)-28c_2(\F)^2 c_4(\F)\\
& -13c_1(\F)c_3(\F)c_4(\F)+c_4(\F)^2-234c_1(\F)^3 c_5(\F) -47c_1(\F)c_2(\F)c_5(\F)\\
& +36c_3(\F)c_5(\F)+1444c_1(\F)^2 c_6(\F)+138c_2(\F)c_6(\F).
\end{aligned}$
\end{enumerate}
\end{lemma}
\begin{proof}
See Out(51), Out(53), Out(55), Out(57), Out(59), Out(61), Out(63), Out(65), Out(76), Out(78), Out(80), Out(82), Out(84), Out(86), Out(88), Out(90) in \cite{mc6}. 
\end{proof}

\begin{lemma}
\label{w7}
Let $\F$ be a rank $7$ vector bundle on a smooth variety $X$. Then,
\begin{enumerate}
\item $c_1(\Lambda^2 \F)=6c_1(\F)$.
\item $c_2(\Lambda^2 \F)=15c_1(\F)^2+5c_2(\F)$.
\item $c_3(\Lambda^2 \F)=20c_1(\F)^3+25c_1(\F)c_2(\F)+3c_3(\F)$.
\item $c_4(\Lambda^2 \F)=15c_1(\F)^4+50c_1(\F)^2 c_2(\F)+10c_2(\F)^2+16c_1(\F)c_3(\F)-c_4(\F)$.
\item $\begin{aligned}[t] c_5(\Lambda^2 \F) = & \ 6 c_1(\F)^5+50c_1(\F)^3 c_2(\F)+40c_1(\F)c_2(\F)^2+34c_1(\F)^2 c_3(\F)+ 12c_2(\F)c_3(\F)\\
& -c_1(\F)c_4(\F)-9c_5(\F).
\end{aligned}$
\item $\begin{aligned}[t]
c_6(\Lambda^2 \F)=& \ c_1(\F)^6+25c_1(\F)^4 c_2(\F)+60c_1(\F)^2 c_2(\F)^2+10c_2(\F)^3+36c_1(\F)^3 c_3(\F)\\
& +52c_1(\F)c_2(\F)c_3(\F)+2c_3(\F)^2+5c_1(\F)^2 c_4(\F)-34c_1(\F)c_5(\F)-25c_6(\F).
\end{aligned}$
\item $\begin{aligned}[t]
c_7(\Lambda^2 \F)=& \ 5c_1(\F)^5 c_2(\F)+40c_1(\F)^3 c_2(\F)^2+30c_1(\F)c_2(\F)^3+19c_1(\F)^4 c_3(\F)\\
& +84c_1(\F)^2 c_2(\F)c_3(\F)+18c_2(\F)^2 c_3(\F)+12c_1(\F)c_3(\F)^2+11c_1(\F)^3 c_4(\F)\\
& +12c_1(\F)c_2(\F)c_4(\F)-6c_3(\F)c_4(\F)-51c_1(\F)^2 c_5(\F)-18c_2(\F)c_5(\F)\\
& -99c_1(\F)c_6(\F)-57c_7(\F).\\
\end{aligned}$ 
\item $\begin{aligned}[t]
c_8(\Lambda^2 \F)=& \ 10c_1(\F)^4 c_2(\F)^2+30c_1(\F)^2 c_2(\F)^3+5c_2(\F)^4+4c_1(\F)^5 c_3(\F)+60c_1(\F)^3 c_2(\F)c_3(\F)\\
& +60c_1(\F)c_2(\F)^2 c_3(\F)+24c_1(\F)^2 c_3(\F)^2+6c_2(\F)c_3(\F)^2+8c_1(\F)^4 c_4(\F)\\
& +33c_1(\F)^2 c_2(\F)c_4(\F)+6c_2(\F)^2 c_4(\F)-9c_1(\F)c_3(\F)c_4(\F)-9c_4(\F)^2\\
& -38c_1(\F)^3 c_5(\F)-51c_1(\F)c_2(\F)c_5(\F)-11c_3(\F)c_5(\F)-162c_1(\F)^2 c_6(\F)\\
& -46c_2(\F)c_6(\F)-228c_1(\F)c_7(\F).\\
\end{aligned}$. 
\item $c_1(\Lambda^3 \F)=15c_1(\F)$.
\item $c_2(\Lambda^3 \F)=105c_1(\F)^2+10c_2(\F)$.
\item $c_3(\Lambda^3 \F)=455c_1(\F)^3+140c_1(\F)c_2(\F)+2c_3(\F)$.
\item $c_4(\Lambda^3 \F)=1365c_1(\F)^4+910c_1(\F)^2 c_2(\F)+45c_2(\F)^2+32c_1(\F)c_3(\F)-8c_4(\F)$.
\item $\begin{aligned}[t]
c_5(\Lambda^3 \F)=& \ 3003c_1(\F)^5+3640c_1(\F)^3 c_2(\F)+585c_1(\F)c_2(\F)^2+234c_1(\F)^2 c_3(\F)+18c_2(\F)c_3(\F)\\
& -102c_1(\F)c_4(\F)-10c_5(\F).
\end{aligned}$
\item $\begin{aligned}[t]
c_6(\Lambda^3 \F)=& \ 5005c_1(\F)^6+10010c_1(\F)^4 c_2(\F)+3510c_1(\F)^2 c_2(\F)^2+120c_2(\F)^3+1040c_1(\F)^3 c_3(\F)\\
& +270c_1(\F)c_2(\F)c_3(\F)-3c_3(\F)^2-600c_1(\F)^2 c_4(\F)-60c_2(\F)c_4(\F)-140c_1(\F)c_5(\F) \\
& +40c_6(\F).\\
\end{aligned}$
\item $\begin{aligned}[t]
c_7(\Lambda^3 \F)=& \ 6435c_1(\F)^7+20020c_1(\F)^5 c_2(\F)+12870c_1(\F)^3 c_2(\F)^2+1440c_1(\F)c_2(\F)^3\\
& +3146c_1(\F)^4 c_3(\F)+1836c_1(\F)^2 c_2(\F)c_3(\F)+72c_2(\F)^2 c_3(\F)-27c_1(\F)c_3(\F)^2\\
& -2156c_1(\F)^3 c_4(\F)-702c_1(\F)c_2(\F)c_4(\F)-18c_3(\F)c_4(\F)-904c_1(\F)^2 c_5(\F)\\
& -72c_2(\F)c_5(\F)+454c_1(\F)c_6(\F)+302c_7(\F).\\
\end{aligned}$
\item $\begin{aligned}[t]
c_8(\Lambda^3 \F)=& \ 6435c_1(\F)^8+30030c_1(\F)^6 c_2(\F)+32175c_1(\F)^4 c_2(\F)^2+7920c_1(\F)^2 c_2(\F)^3\\
& +210c_2(\F)^4+6864c_1(\F)^5 c_3(\F)+7524c_1(\F)^3 c_2(\F)c_3(\F)+1008c_1(\F)c_2(\F)^2 c_3(\F)\\
& -84c_1(\F)^2 c_3(\F)^2-24c_2(\F)c_3(\F)^2-5280c_1(\F)^4 c_4(\F)-3759c_1(\F)^2 c_2(\F)c_4(\F)\\
& -192c_2(\F)^2 c_4(\F)-237c_1(\F)c_3(\F)c_4(\F)+6c_4(\F)^2-3570c_1(\F)^3 c_5(\F)\\
& -951c_1(\F)c_2(\F)c_5(\F)+33c_3(\F)c_5(\F)+2370c_1(\F)^2 c_6(\F)+222c_2(\F)c_6(\F)\\
\end{aligned}$

$\begin{aligned}[t]
& \hskip 1.4cm + 3624 c_1(\F) c_7(\F).\\
\end{aligned}$
\end{enumerate}
\end{lemma}
\begin{proof}
See Out(11), Out(13), Out(15), Out(17), Out(19), Out(21), Out(23), Out(25), Out(38), Out(40), Out(42), Out(44), Out(46), Out(48), Out(50), Out(52) in \cite{mc7}.
\end{proof}

\section{General Riemann-Roch calculations}
\label{app2}

We first compute the Todd class of a given variety up to degree $8$.

\begin{lemma} 
\label{td}
Let $X$ be a smooth irreducible variety and set $c_i=c_i(X)$. Then we have:

$\begin{aligned}[t] 
{\rm Td}(T_X) &=1+\frac{1}{2}c_1+\frac{1}{12}(c_1^2+c_2)+\frac{1}{24}c_1c_2-\frac{1}{720}(c_1^4-4c_1^2c_2-3c_2^2-c_1c 3+c_4)\\
& -\frac{1}{1440}(c_1^3c_2-3c_1c_2^2-c_1^2c_3+c_1c_4)\\
& +\frac{1}{60480}(2c_1^6-12c_1^4c_2+11c_1^2c_2^2+10c_2^3+5c_1^3c_3+11c_1c_2c_3-c_3^2-5c_1^2c_4-9c_2c_4-2c_1c_5\\
& \hskip 1.6cm +2c_6)\\
& +\frac{1}{120960}(2c_1^5c_2-10c_1^3c_2^2+10c_1c_2^3-2c_1^4c_3+11c_1^2c_2c_3-c_1c_3^2+2c_1^ 3c_4-9c_1c_2c_4-2c_1^2c_5\\
& \hskip 1.8cm +2c_1c_6)\\
 & -\frac{1}{3628800}(3c_1^8-24c_1^6c_2+50c_1^4c_2^2-8c_1^2c_2^3-21c_2^4+14c_1^5c_3-26c_1^3c_2c_3-50c_1c_2^2c_3-3c_1^2c_3^2\\
& \hskip 1.9cm +8c_2c_3^2-14c_1^4c_4+19c_1^2c_2c_4+34c_2^2c_4+13c_1c_3c_4-5c_4^2+7c_1^3c_5+16c_1c_2c_5\\
& \hskip 1.9cm -3c_3c_5-7c_1^2c_6-13c_2c_6-3c_1c_7+3c_8) + \ldots
\end{aligned}$
\end{lemma}
\begin{proof}
See \cite[Out(40)]{mc6}. 
\end{proof}

Next, we compute the Chern character of a vector bundle, up to degree $8$.

\begin{lemma} 
\label{ch}
Let $X$ be a smooth irreducible variety, let $\F$ be a rank $r$ vector bundle on $X$ and set $d_i=c_i(\F)$. Then we have:

$\begin{aligned}[t] 
{\rm Ch}(\F) & = r+d_1+\frac{1}{2}(d_1^2-2d_2)+\frac{1}{6}(d_1^3-3d_1d_2+3d_3)+\frac{1}{24}(d_1^4-4d_1^2d_2+4d_1d_3+2d_2^2-4d_4)\\
& +\frac{1}{120}(d_1^5-5d_1^3d_2+5d_1d_2^2+5d_1^2d_3-5d_2d_3-5d_1d_4+5d_5)\\
& +\frac{1}{720}(d_1^6-6d_1^4d_2+9d_1^2d_2^2-2d_2^3+6d_1^3d_3-12d_1d_2d_3+3d_3^2-6d_1^2d_4+6d_2d_4+6d_1d_5-6 d_6)\\
& +\frac{1}{5040}(d_1^7-7d_1^5d_2+14d_1^3d_2^2-7d_1d_2^3+7d_1^4d_3-21d_1^2d_2d_3+7d_2^2d_3+7d_1d_3^2 -7d_1^3 d_4\\
& \hskip 1.4cm +14d_1d_2d_4-7d_3d_4+7d_1^2d_5-7d_2d_5-7d_1d_6+7d_7)\\
& +\frac{1}{40320}(d_1^8-8d_1^6d_2+20d_1^4d_2^2-16d_1^2d_2^3+2d_2^4+8d_1^5d_3-32d_1^3d_2d_3+
24d_1d_2^2d_3+12d_1^2d_3^2\\
& \hskip 1.4cm -8d_2d_3^2-8d_1^4d_4+24d_1^2d_2d_4-8d_2^2d_4-16d_1d_3d_4+4d_4^2+8d_1^3d_5- 16d_1d_2d_5+8d_3d_5\\
& \hskip 1.4cm -8d_1^2d_6+8d_2d_6+8d_1d_7-8d_8)+\cdots
\end{aligned}$
\end{lemma}
\begin{proof}
See \cite[Out(103)]{mc6}. 
\end{proof}

\begin{lemma} 
\label{rr6}
Let $X$ be a smooth irreducible variety of dimension $6$ and let $\F$ be a rank $6$ vector bundle on $X$. Set $c_i=c_i(X)$ and $d_i=c_i(\F)$. Then we have:

$\begin{aligned}[t] 
\chi(\F) & =\frac{1}{10080}(2c_1^6-2c_1c_5+2c_6+11c_1^2c_2^2+11c_1c_2c_3-c_3^2)-\frac{1}{840}c_1^4c_2+\frac{1}{2016}(2c_2^3+c_1^3c_3-c_1^2c_4)\\
& -\frac{1}{1120}c_2c_4-\frac{1}{1440}(c_1^3c_2d_1-c_1^2c_3d_1+c_1c_4d_1+c_1^4 d_1^2-c_1c_3d_1^2+c_4d_1^2)\\
& +\frac{1}{480}(c_1c_2^2d_1+c_2^2d_1^2+2c_1 d_1^5-2c_2^2d_2+2d_3^2)+\frac{1}{288}(2c_1c_2d_1^3+c_1^2d_1^4+c_2d_1^4+2c_1^2d_2^2+2c_2d_2^2)\\
\end{aligned}$

$\begin{aligned}[t] 
& \hskip 1.1cm +\frac{1}{720}(d_1^6-2d_2^3+c_1^4d_2-c_1c_3d_2+c_4d_2-4c_1^2c_2d_2+2c_1^2c_2d_1^2)\\
& \hskip 1.1cm -\frac{1}{48}(c_1c_2d_1d_2+c_1d_1^3d_2-c_1d_1d_2^2-c_1c_2d_3-c_1d_1^2d_3+c_1d_2d_3+c_1d_1d_4-c_1d_5)\\
& \hskip 1.1cm +\frac{1}{72}(-c_1^2d_1^2d_2-c_2d_1^2d_2+c_1^2d_1d_3+c_2d_1d_3-c_1^2d_4-c_2d_4)\\
& \hskip 1.1cm +\frac{1}{80}d_1^2d_2^2-\frac{1}{120}(2d_1d_2d_3+d_1^4d_2-d_1^3d_3+d_1^2d_4-d_2d_4-d_1d_5+d_6).
\end{aligned}$
\end{lemma}
\begin{proof}
Follows from Lemmas \ref{ch}, \ref{td} and Riemann-Roch (see \cite[Out(70)]{mc4}). 
\end{proof}

\begin{lemma}
\label{rr10}
Let $X$ be a smooth irreducible variety of dimension $6$ and let $\F$ be a rank $10$ vector bundle
on $X$. Set $c_i = c_i(X)$ and $d_i = c_i(\F)$. Then we have:

$\begin{aligned}[t]
\chi(\F) & =\frac{1}{120} (d_1d_5-d_6-d_1^2 d_4+d_2d_4+d_1^3 d_3-d_1^4 d_2)\\
& +\frac{1}{48}(-c_1d_1d_4+c_1d_5-c_1d_2d_3+c_1d_1^2 d_3+c_1c_2d_3+c_1d_1d_2^2-c_1d_1^3 d_2-c_1c_2d_1d_2)\\
& +\frac{1}{72}(-c_1^2 d_4-c_2d_4+c_2d_1d_3+c_1^2 d_1d_3-c_2d_1^2 d_2-c_1^2 d_1^2 d_2)\\
& +\frac{1}{240}(d_3^2-c_2^2 d_2+c_1d_1^5)-\frac{1}{60}d_1d_2d_3+\frac{1}{360}(-d_2^3+c_1^2 c_2d_1^2)+\frac{1}{80}d_1^2 d_2^2\\
& +\frac{1}{144} (c_1^2 d_2^2+c_2d_2^2+c_1c_2d_1^3)+\frac{1}{720}(-c_1c_3d_2+c_4d_2+d_1^6+c_1^4 d_2) \\
& -\frac{1}{180}c_1^2 c_2d_2+\frac{1}{288}(c_1^2 d_1^4+c_2d_1^4)+\frac{1}{1440} (c_1c_3d_1^2-c_4d_1^2+c_1^2 c_3d_1-c_1c_4d_1-c_1^4 d_1^2-c_1^3 c_2d1)\\
& +\frac{1}{480} (c_2^2 d_1^2+c_1c_2^2 d_1)+\frac{1}{3024}(c_1^6+5c_2^3-c_1c_5+c_6)-\frac{1}{672}c_2c_4\\
& +\frac{1}{6048} (5c_1^3 c_3+11c_1c_2c_3-c_3^2-5c_1^2 c_4+11c_1^2 c_2^2)-\frac{1}{504}c_1^4 c_2
\end{aligned}$
\end{lemma}
\begin{proof}
Follows from Lemmas \ref{ch}, \ref{td} and Riemann-Roch (see \cite[Out(10)]{mc5}).
\end{proof}

\begin{lemma} 
\label{chiw24}
Let $X$ be a smooth irreducible variety of dimension $6$ and let $\F$ be a rank $4$ vector bundle on $X$. Set $c_i=c_i(X)$ and $f_i=c_i(\F)$. Then we have:
\begin{equation*}
\begin{split}
\chi((\Lambda^2 \F)(t)) & = \frac{H^6}{120}t^6+\frac{H^5}{40}(c_1+f_1)t^5+\frac{H^4}{48}(c_1^2+c_2+3c_1f_1+ 3f_1^2-4f_2)t^4\\
& +\frac{H^3}{24}(c_1c_2+c_1^2 f_1+c_2f_1+3c_1f_1^2+2f_1^3-4c_1f_2-4f_1f_2)t^3\\
& +\frac{H^2}{240}(-c_1^4+4c_1^2 c_2+3c_2^2+c_1c_3-c_4+15c_1c_2f_1+15c_1^2 f_1^2+15c_2f_1^2+30c_1f_1^3 +15f_1^4\\
& \hskip 1.2cm -20c_1^2 f_2-20c_2f_2-60c_1f_1f_2-40f_1^2 f_2+20f_2^2-20f_1f_3+80f_4)t^2\\
& +\frac{H}{240}(-c_1^3 c_2+3c_1c_2^2+c_1^2 c_3-c_1c_4-c_1^4 f_1+4c_1^2 c_2f_1+3c_2^2 f_1+
c_1c_3f_1-c_4f_1+15c_1c_2f_1^2\\
& \hskip 1.2cm +10c_1^2 f_1^3+10c_2f_1^3+15c_1f_1^4+6f_1^5-20c_1c_2f_2-20c_1^2 f_1f_2-20c_2f_1f_2-40c_1f_1^2 f_2\\
& \hskip 1.2cm -20f_1^3 f_2+20c_1f_2^2+20f_1f_2^2-20c_1f_1f_3-20f_1^2 f_3+80c_1f_4+80f_1f_4)t\\
& +\frac{1}{10080}(2c_1^6-12c_1^4 c_2+11c_1^2 c_2^2+10c_2^3+5c_1^3 c_3+11c_1c_2c_3-c_3^2-5c_1^2 c_4-9c_2 c_4-2c_1c_5+2c_6\\
& \hskip 1.2cm - 21c_1^3 c_2f_1+63c_1c_2^2 f_1+21c_1^2 c_3f_1-21c_1c_4f_1-21c_1^4 f_1^2+84c_1^2 c_2f_1^2+
63c_2^2 f_1^2+21c_1c_3f_1^2\\
& \hskip 1.2cm -21 c_4f_1^2+210c_1c_2f_1^3+105c_1^2 f_1^4+105c_2f_1^4+126c_1f_1^5+42f_1^6+28c_1^4 f_2-112c_1^2 c_2f_2\\
& \hskip 1.2cm -84c_2^2 f_2-28c_1c_3f_2+28c_4f_2-420c_1c_2f_1f_2-280c_1^2 f_1^2 f_2-280c_2f_1^2 f_2-420c_1f_1^3 f_2\\
& \hskip 1.2cm -168f_1^4 f_2+140c_1^2 f_2^2 +140c_2f_2^2+420c_1f_1f_2^2+252f_1^2 f_2^2-56f_2^3-140c_1^2 f_1f_3\\
& \hskip 1.2cm -140c_2f_1f_3-420c_1f_1^2 f_3-252f_1^3 f_3+
84f_1f_2f_3+84f_3^2+560c_1^2 f_4+560c_2f_4\\
& \hskip 1.2cm +1680c_1f_1f_4+1092f_1^2 f_4-672f_2f_4).
\end{split}
\end{equation*}
\end{lemma}
\begin{proof}
Follows from Lemma \ref{rr6} (see \cite[Out(91)]{mc4}). 
\end{proof}

\begin{lemma} 
\label{chiw25}
Let $X$ be a smooth irreducible variety of dimension $6$ and let $\F$ be a rank $5$ vector bundle on $X$. Set $c_i=c_i(X)$ and $f_i=c_i(\F)$. Then we have:
\begin{equation*}
\begin{split}
\chi((\Lambda^2 \F)(t)) & = \frac{H^6}{72}t^6+\frac{H^5}{120}(5c_1+4f_1)t^5+\frac{H^4}{144}(5c_1^2+5c_2+12c_1 f_1+12f_1^2-18f_2)t^4\\ 
& +\frac{H^3}{72}(5c_1c_2+4c_1^2 f_1+4c_2f_1+12c_1f_1^2+8f_1^3-18c_1f_2-18f_1f_2+6f_3)t^3\\
& +\frac{H^2}{144}(-c_1^4+4c_1^2 c_2+3c_2^2+c_1c_3-c_4+12c_1c_2f_1+12c_1^2 f_1^2+12c_2f_1^2+24c_1f_1^3 +12f_1^4-18c_1^2 f_2\\
& \hskip 1.2cm -18c_2f_2-54c_1f_1f_2-36f_1^2 f_2+18f_2^2+18c_1f_3+36f_4)t^2\\
& +\frac{H}{720}(-5c_1^3 c_2+15c_1c_2^2+5c_1^2 c_3-5c_1c_4-4c_1^4 f_1+16c_1^2 c_2f_1+12c_2^2 f_1+4c_1c_3 f_1-4c_4f_1\\
& \hskip 1.2cm +60c_1c_2f_1^2+40c_1^2 f_1^3+40c_2f_1^3+60c_1f_1^4+24f_1^5-90c_1c_2f_2-90c_1^2 f_1f_2-90c_2 f_1f_2\\
& \hskip 1.2cm -180c_1f_1^2 f_2-90f_1^3 f_2+90c_1f_2^2+90f_1f_2^2+30c_1^2 f_3+30c_2f_3-30f_1^2 f_3-30f_2 f_3\\
& \hskip 1.2cm +180c_1 f_4+210f_1f_4-330f_5)t\\
& +\frac{1}{30240}(10c_1^6-60c_1^4 c_2+55c_1^2 c_2^2+50c_2^3+25c_1^3 c_3+55c_1c_2c_3-5c_3^2-25c_1^2 c_4-45c_2c_4-10c_1c_5\\
& \hskip 1.2cm +10c_6-84c_1^3 c_2f_1+252c_1c_2^2 f_1+84c_1^2 c_3f_1-84c_1c_4f_1-84c_1^4 f_1^2+ 336c_1^2 c_2f_1^2+252c_2^2 f_1^2\\
& \hskip 1.2cm +84c_1c_3f_1^2-84c_4f_1^2+840c_1c_2f_1^3+420c_1^2 f_1^4+420c_2f_1^4+ 
504c_1f_1^5+168f_1^6+126c_1^4 f_2\\
& \hskip 1.2cm -504c_1^2 c_2f_2-378c_2^2 f_2-126c_1c_3f_2+126c_4f_2-1890c_1c_2f_1f_2 -1260c_1^2 f_1^2f_2\\
& \hskip 1.2cm -1260c_2f_1^2 f_2-1890c_1f_1^3 f_2-756f_1^4 f_2+630c_1^2 f_2^2+630c_2f_2^2+1890c_1f_1 f_2^2+1134f_1^2 f_2^2\\
& \hskip 1.2cm -252f_2^3+630c_1c_2f_3-630c_1f_1^2 f_3-504f_1^3f_3-630c_1f_2f_3-252f_1f_2f_3+378f_3^2\\
& \hskip 1.2cm +1260c_1^2f_4+1260c_2f_4+4410c_1f_1f_4+3024f_1^2 f_4-1764f_2f_4- 6930c_1f_5-5544f_1f_5).
\end{split}
\end{equation*}
\end{lemma}
\begin{proof}
Follows from Lemma \ref{rr10} (see \cite[Out(31)]{mc5}). 
\end{proof}

\section{Riemann-Roch calculations on hypersurfaces}
\label{app3}

We now perform the necessary calculations on hypersurfaces used in the proof of Theorem \ref{main}.

\begin{lemma}
\label{suz4}
Let $\E$ be an Ulrich bundle of rank $4$ on a smooth hypersurface $X \subset \P^7$ of degree $d$. Then the following holds:
\begin{enumerate}
\item $\begin{aligned}[t]
\chi((\Lambda^2 \E)(m-2d+2))=& \frac{d}{120}m^6-\frac{d}{40}(-10+3d)m^5+\frac{5d}{72}(44-27d+4d^2)m^4\\
& -\frac{d}{72}(-1400+1320d-400d^2+39d^3)m^3\\
& +\frac{d}{360}(24419-31500d+14670d^2-2925d^3+208d^4)m^2\\
& -\frac{d}{360}(-44190+73257d-46700d^2+14310d^3-2080d^4+111d^5)m\\
& +\frac{d}{340200}(30562169-62639325d+51356676d^2-21546000d^3\\
& \hskip 1.7cm +4812171d^4-524475d^5+19984d^6).
\end{aligned}$
\end{enumerate}
\end{lemma}
\begin{proof}
(1) is obtained by Lemma \ref{chiw24}, replacing $t=m-2d+2$, the values of $c_i$ given in Lemma \ref{xn} and of $f_i$ given in Lemma \ref{xne} (see \cite[Out(100)]{mc4}). 
\end{proof}

\begin{lemma}
\label{suz5}
Let $\E$ be an Ulrich bundle of rank $5$ on a smooth hypersurface $X \subset \P^7$ of degree $d$. Then the following hold:
\begin{enumerate}
\item $\begin{aligned}[t]
\chi((\Lambda^2 \E)(m-\frac{5}{2}(d-1))& = \frac{d}{72}m^6-\frac{d}{24}(-11+4d)m^5+\frac{5d}{576}(713-528d+95d^2)m^4\\
\end{aligned}$

$\begin{aligned}[t]
& \hskip 2.8cm -\frac{5d}{288}(-2519+2852d-1045d^2+124d^3)m^3\\
& \hskip 2.8cm +\frac{d}{576}(98122-151140d+84675d^2-20460d^3+1795d^4)m^2\\
& \hskip 2.8cm -\frac{d}{576}(-199551+392488d-299200d^2+110540d^3-19745d^4+1356d^5)m\\
& \hskip 2.8cm +\frac{d}{1548288}(444410639-1072786176d+1044516123d^2-525127680d^3\\
& \hskip 4.7cm +143409693d^4-20047104d^5+ 1107385d^6).
\end{aligned}$
\item $\begin{aligned}[t]
\chi((\Lambda^3 \E)(m-\frac{5}{2}(d-1)))=& \frac{d}{72}m^6-\frac{d}{24}(-10+3d)m^5+\frac{5d}{576}(587-360d+53d^2)m^4\\
& -\frac{5d}{288}(-10+3d)(187-120d+17d^2)m^3\\
& +\frac{d}{576}(65362-84150d+38895d^2-7650d^3+535d^4)m^2\\
& -\frac{d}{288}(-10+3d)(5931-8025d+3790d^2-735d^3+47d^4)m\\
& +\frac{d}{1548288}(234265319-478275840d+388398675d^2-160473600d^3\\
& \hskip 1.9cm +35211813d^4-3790080d^5+146593d^6).
\end{aligned}$
\end{enumerate}
\end{lemma}
\begin{proof}
(1) is obtained by Lemma \ref{chiw25}, replacing $t=m-\frac{5}{2}(d-1)$, the values of $c_i$ given in Lemma \ref{xn} and of $f_i$ given in Lemma \ref{xne} (see \cite[Out(49)]{mc5}). (2) is obtained similarly from Lemma \ref{chiw25} using the fact that $\chi((\Lambda^3 \E)(t)=\chi((\Lambda^2 \E^*)(t+\frac{5}{2}(d-1))$ (see \cite[Out(50)]{mc5}). 
\end{proof}

\begin{lemma}
\label{suz6}
Let $\E$ be an Ulrich bundle of rank $6$ on a smooth hypersurface $X \subset \P^9$ of degree $d$. Then the following hold:
\begin{enumerate}
\item $\begin{aligned}[t]
& \chi(\Lambda^2 \E(m-3d+3)) = \frac{d}{2688}m^8-\frac{d}{672}(-14+5d)m^7+\frac{d}{960}(483-350d+62d^2)m^6\\
& -\frac{d}{960}(-14+5d)(469-350d+61d^2)m^5\\
& +\frac{d}{1920}(109837-164150d+89840d^2-21350d^3+1858d^4)m^4\\
& -\frac{d}{960}(-14 +5d)(20657-31850d+17705d^2-4200d^3+358d^4)m^3\\
& +\frac{d}{6720}(6549514-15182895d+14302806d^2-7010675d^3+1884820d^4-263130d^5+14870d^6)m^2\\
& -\frac{d}{13440}(-14+5d)(1699080-4071410d+3926321d^2-1949220d^3+524314d^4-72170d^5\\
& \hskip 3.4cm +3965d^6)m\\
& +\frac{d}{169344000}(233706519541-749294280000d+1023683569750d^2-778550661000d^3\\
& \hskip 2.4cm +360297139573d^4- 103729374000d^5+18104141400d^6-1748565000d^7\\
& \hskip 2.4cm +71669736d^8).
\end{aligned}$
\item $\begin{aligned}[t]
& \chi(\Lambda^3 \E(m-3d+3)) = \frac{d}{2016}m^8-\frac{d}{504}(-13+4d)m^7+\frac{d}{2160}(1247-780d+118d^2)m^6\\
& \hskip .8cm -\frac{d}{360}(-13+4d)(201-130d+19d^2)m^5\\
& \hskip .8cm +\frac{d}{4320}(241996-313560d+147155d^2-29640d^3+2154d^4)m^4\\
& \hskip .8cm -\frac{d}{2160}(-13+4d)(45111-60580d+28805d^2-5720d^3+394d^4)m^3\\
\end{aligned}$

$\begin{aligned}[t]
& +\frac{d}{30240}(24379978-49261212d+39993401d^2-16691220d^3+3762129d^4-430248d^5\\
& \hskip 1.5cm +19032d^6)m^2\\
& -\frac{d}{30240}(-13+4d)(3116229-6542692d+5444216d^2-2290964d^3+509035d^4\\
& \hskip 1.5cm -55224d^5+2040d^6)m\\
& +\frac{d}{508032000}(483969803049-1361168827200d+1612701345950d^2-1050469056000d^3\\
& \hskip 2.4cm +409833928497d^4- 97149124800d^5+13318661400d^6-891072000d^7\\
& \hskip 2.4cm +14981104d^8).
\end{aligned}$
\item $\begin{aligned}[t]
& \chi(\Lambda^4 \E(m-3d+3)) = \frac{d}{2688}m^8-\frac{d}{224}(-4+d)m^7+\frac{d}{960}(353-180d+22d^2)m^6\\
& -\frac{3d}{320}(-4+d)(113-60d+7d^2)m^5+\frac{d}{1920}(57317-61020d+23300d^2-3780d^3+218d^4)m^4\\
& -\frac{d}{320}(-4+d)(10517-11700d+4535d^2-720d^3+38d^4)m^3\\
& +\frac{d}{3360}(1185579-1987713d+1324764d^2-448875d^3+80843d^4-7182d^5+239d^6)m^2\\
& -\frac{d}{4480}(-4+d)(590076-1038060d+713205d^2-243720d^3+42698d^4-3420d^5+101d^6)m\\
& +\frac{d}{169344000}(56633150341-133829236800d+131659211350d^2-70341793200d^3\\
& \hskip 2.4cm +22114878373d^4-4099183200 d^5+425383800d^6-22906800d^7+656136d^8).
\end{aligned}$
\end{enumerate}
\end{lemma}
\begin{proof}
(1) is obtained by the expression of $\chi((\Lambda^2 \F)(t)$ for a rank $6$ bundle $\F$ (see \cite[Out(111)]{mc6}), replacing $t=m-3(d-1)$, the values of $c_i$ given in Lemma \ref{xn} and of $f_i$ given in Lemma \ref{xne} (see \cite[Out(138)]{mc6}). (2) is obtained by the expression of $\chi((\Lambda^3 \F)(t)$ for a rank $6$ bundle $\F$ (see \cite[Out(112)]{mc6}), replacing $t=m-3(d-1)$, the values of $c_i$ given in Lemma \ref{xn} and of $f_i$ given in Lemma \ref{xne} (see \cite[Out(139)]{mc6}). 
(3) is obtained using the expression of $\chi((\Lambda^2 \F)(t)$ and the fact that $\chi((\Lambda^4 \E)(t)=\chi((\Lambda^2 \E^*)(t+3(d-1))$ (see \cite[Out(140)]{mc6}).
\end{proof}

\begin{lemma}
\label{suz7}
Let $\E$ be an Ulrich bundle of rank $7$ on a smooth hypersurface $X \subset \P^9$ of degree $d$. Then the following hold:
\begin{enumerate}
\item $\begin{aligned}[t]
& \chi((\Lambda^2 \E)(m-\frac{7}{2}(d-1))) = \frac{d}{1920}m^8-\frac{d}{160}(-5+2d)m^7+\frac{7d}{8640}(-20+7d)(-50+23d)m^6\\
& -\frac{7d}{960}(-5+2d)(325-270d+53d^2)m^5\\
& +\frac{7d}{138240}(2110467-3510000d+2146690d^2-572400d^3+56147d^4)m^4\\
& -\frac{7d}{23040}(-5+2d)(400467-684000d+424090d^2-113040d^3+10931d^4)m^3\\
& +\frac{d}{138240}(294927561-756882630d+792886542d^2-434114100d^3+131018209d^4\\
& \hskip 1.7cm -20659590d^5+1328936d^6)m^2\\
& -\frac{d}{23040}(-5+2d)(19408518-51222105d+54741054d^2-30313350d^3+9168002d^4\\
& \hskip 1.6cm -1434465d^5+90682d^6)m\\
& +\frac{d}{143327232000}(513397845100961-1811047631616000d+2735536296233740d^2\\
& \hskip 2.9cm -2311436590848000d^3+1194935635595478d^4-386871738624000d^5\\
& \hskip 2.9cm +76557801497260d^6-8461718784000d^7+399973316561d^8).
\end{aligned}$
\item $\begin{aligned}[t]
& \chi((\Lambda^3\E)(m-\frac{7}{2}(d-1)))=\frac{d}{1152}m^8-\frac{d}{288}(-14+5d)m^7+\frac{7d}{1728}(290-210d+37d^2)m^6\\ 
& -\frac{7d}{288}(-14+5d)(47-35d+6d^2)m^5\\
& +\frac{7d}{138240}(2647681-3948000d+2146070d^2-504000d^3+43089d^4)m^4\\
& -\frac{7d}{69 120}(-14+5d)(499521-767200d+421670d^2-98000d^3+8089d^4)m^3\\
& +\frac{d}{414720}(954207685-2202887610d+2057673702d^2-995204700d^3+262468605d^4\\
& \hskip 1.7cm -35672490d^5+1939448d^6)m^2\\
& -\frac{d}{414720}(-14+5d)(124417969-296353470d+282424737d^2-137577300d^3+35979531d^4\\
& \hskip 1.7cm -4750830d^5+242723d^6)m\\
& +\frac{d}{8957952000}(29526126063793-94059984564000d+127169755078220d^2-95268653172000d^3+\\
& \hskip 2.6cm 43183463113014d^4-12086573436000d^5+2026225100780d^6-183498588000d^7\\
& \hskip 2.6cm +6668724193d^8).
\end{aligned}$
\item $\begin{aligned}[t]
& \chi((\Lambda^4\E)(m-\frac{7}{2}(d-1)))= \frac{d}{1152}m^8-\frac{d}{288}(-13+4d)m^7+\frac{7d}{3456}(499-312d+47d^2)m^6\\
& -\frac{7d}{1152}(-7+3d)(-13+4d)(-23+5d)m^5\\
& +\frac{7d}{34560}(485239-627900d+293180d^2-58500d^3+4191d^4)m^4\\
& -\frac{7d}{17280}(-13+4d)(90624-121420d+57245d^2-11180d^3+751d^4)m^3\\
& +\frac{d}{51840}(73648124-148442112d+119778477d^2-49475790d^3+10983966d^4-1230138d^5\\
& \hskip 1.6cm +53053d^6)m^2\\
& -\frac{d}{103680}(-13+4d)(18886837-39510588d+32595240d^2-13514280d^3+2935833d^4-308412d^5\\
& \hskip 1.7cm +11210d^6)m\\
& +\frac{d}{4478976000}(7572278446559-21213695318400d+24947874489460d^2-16064770176000d^3\\
& \hskip 2.5cm +6166188349482d^4-1429690953600d^5+190927651540d^6-12591072000d^7\\
& \hskip 2.5cm +242742959d^8).
\end{aligned}$  
\item $\begin{aligned}[t]
& \chi((\Lambda^5\E)(m-\frac{7}{2}(d-1)))= \frac{d}{1920}m^8-\frac{d}{160}(-4+d)m^7+\frac{7d}{17280}(1271-648d+79d^2)m^6\\
& -\frac{7d}{1920}(-4+d)(407-216d+25d^2)m^5\\
& +\frac{7d}{34560}(206553-219780d+83680d^2-13500d^3+773d^4)m^4\\
& -\frac{7d}{5760}(-4+d)(37929-42156d+16261d^2-2556d^3+134d^4)m^3\\
& +\frac{d}{17280}(8560242-14337162d+9522975d^2-3207330d^3+573356d^4-50652d^5+1687d^6)m^2\\
& -\frac{d}{11520}(-4+d)(2133108-3746844d+2561847d^2-867816d^3+150574d^4-12060d^5\\
& \hskip 1.6cm +359d^6)m\\
\end{aligned}$ 

$\begin{aligned}[t]
& +\frac{d}{143327232000}(67498793060561-159235658956800d+156004224862540d^2\\
& \hskip 2.9cm -82788720537600d^3+25821414047478d^4-4758314803200d^5\\
& \hskip 2.9cm +494189940460d^6-26799206400d^7+743464961d^8).
\end{aligned}$  
\end{enumerate}
\end{lemma}
\begin{proof}
See Out(108)-Out(111) in \cite{mc7}.
\end{proof}

\end{document}